\documentclass[a4paper,11pt]{article}
\usepackage[T1]{fontenc}
\usepackage[latin1]{inputenc}
\usepackage{amsmath,amsfonts}
\usepackage{amsthm}
\usepackage{amssymb}
\usepackage[USenglish]{babel}
\usepackage{amsfonts}
\usepackage{enumerate}
\usepackage{verbatim}
\usepackage{graphicx}
\usepackage[left=3cm, right=3cm, top=3cm, bottom=3cm]{geometry}
\usepackage{tikz}
\usepackage{csquotes}
\usepackage{mathtools}
\usepackage{mathscinet} 
\usepackage{tabularx}
\usepackage{scrextend}
\usepackage{float}
\usepackage{xcolor}
\usepackage{todonotes}

\usepackage[final, pdftex, pdfpagelabels, pdfstartview = {FitH}, bookmarks, plainpages = false, linktoc=all,linkcolor=black, citecolor=blue, urlcolor=blue,
filecolor=black]{hyperref}
\usepackage{cleveref}

\usetikzlibrary{matrix,arrows,shapes,decorations,decorations.pathreplacing}

\theoremstyle{plain}
\newtheorem{theorem}{Theorem}[section]
\newtheorem{lemma}[theorem]{Lemma}
\newtheorem{prop}[theorem]{Proposition}
\newtheorem{cor}[theorem]{Corollary}

\newcounter{introcounter}

\newtheorem{introthm}[introcounter]{Theorem}

\theoremstyle{definition}
\newtheorem{definition}[theorem]{Definition}

\theoremstyle{remark}
\newtheorem{claim}[theorem]{Claim}

\renewcommand{\tilde}{\widetilde}
\renewcommand{\bar}{\overline}
\newcommand{\bbC}{\mathbb{C}}
\newcommand{\bbD}{\mathbb{D}}
\newcommand{\bbH}{\mathbb{H}}
\newcommand{\bbK}{\mathbb{K}}
\newcommand{\bbN}{\mathbb{N}}
\newcommand{\bbP}{\mathbb{P}}
\newcommand{\bbQ}{\mathbb{Q}}
\newcommand{\bbR}{\mathbb{R}}
\newcommand{\bbZ}{\mathbb{Z}}
\newcommand{\bbLL}{\mathbb{LL}}
\newcommand{\ILL}{\mathrm{ILL}}

\newcommand{\Raw}{\Rightarrow}
\newcommand{\raw}{\rightarrow}
\newcommand{\daw}{\downarrow}
\newcommand{\upaw}{\uparrow}
\newcommand{\Lra}[1]{\xrightarrow[L]{#1}}
\newcommand{\Rra}[1]{\xrightarrow[R]{#1}}
\newcommand{\xra}{\xrightarrow}
\newcommand{\cug}{\subseteq}
\newcommand{\sug}{\supseteq}
\newcommand{\into}{\longmapsto}

\newcommand{\undx}{{\underline{x}}}
\newcommand{\undy}{{\underline{y}}}
\newcommand{\undz}{{\underline{z}}}
\newcommand{\undw}{{\underline{w}}}
\newcommand{\undtheta}{\underline{\theta}}
\newcommand{\undH}{{\underline{\mathbf{H}}}}

\newcommand{\calA}{\mathcal{A}}
\newcommand{\calB}{\mathcal{B}}
\newcommand{\calC}{\mathcal{C}}
\newcommand{\calD}{\mathcal{D}}
\newcommand{\calF}{\mathcal{F}}
\newcommand{\calG}{\mathcal{G}}
\newcommand{\calH}{\mathcal{H}}
\newcommand{\calI}{\mathcal{I}}
\newcommand{\calK}{\mathcal{K}}
\newcommand{\calL}{\mathcal{L}}
\newcommand{\calM}{\mathcal{M}}
\newcommand{\calO}{\mathcal{O}}
\newcommand{\calP}{\mathcal{P}}
\newcommand{\calQ}{\mathcal{Q}}
\newcommand{\calR}{\mathcal{R}}
\newcommand{\calS}{\mathcal{S}}
\newcommand{\calT}{\mathcal{T}}
\newcommand{\calU}{\mathcal{U}}
\newcommand{\calV}{\mathcal{V}}
\newcommand{\calX}{\mathcal{X}}

\newcommand{\bfH}{\mathbf{H}}
\newcommand{\bfP}{\mathbf{P}}
\newcommand{\bfQ}{\mathbf{Q}}
\newcommand{\bfR}{\mathbf{R}}
\newcommand{\bfM}{\mathbf{M}}
\newcommand{\bfS}{\mathbf{S}}
\newcommand{\bfU}{\mathbf{U}}
\newcommand{\bfV}{\mathbf{V}}
\newcommand{\bfN}{\mathbf{N}}

\newcommand{\frh}{\mathfrak{h}}
\newcommand{\frD}{\mathfrak{D}}
\newcommand{\frP}{\mathfrak{P}}
\newcommand{\frQ}{\mathfrak{Q}}
\newcommand{\frR}{\mathfrak{R}}
\newcommand{\frgl}{\mathfrak{gl}}
\newcommand{\frsl}{\mathfrak{sl}}
\newcommand{\frg}{\mathfrak{g}}
\newcommand{\frb}{\mathfrak{b}}

\newcommand{\sbim}{{\mathbb{S}Bim}}
\newcommand{\bsbim}{\mathbb{BS}Bim}

\newcommand{\cols}{{\color{red} s}}
\newcommand{\colt}{{\color{blue} t}}
\newcommand{\colu}{{\color{green} u}}

\newcommand{\lra}{\longrightarrow}
\newcommand*{\longhookrightarrow}{\ensuremath{\lhook\joinrel\relbar\joinrel\rightarrow}}

\newcommand{\geqH}{\overset{H}{\geq}}

\newcommand{\Leo}[1]{\todo[size=\tiny,color=blue!30]{#1 \\ \hfill --- L.}}
\newcommand{\Nico}[1]{\todo[size=\tiny,color=red!30]{#1 \\ \hfill --- N.}}

\newcommand{\Address}{
 \bigskip{\footnotesize

 \textsc{Universidad de Chile, Santiago, Chile}\par\nopagebreak
 \textit{E-mail address}: \texttt{nlibedinsky@gmail.com}
}

 \bigskip{\footnotesize
 \textsc{Albert-Ludwigs-Universit\"at Freiburg, Freiburg im Breisgau, Germany}\par\nopagebreak
\textit{E-mail address}: \texttt{leonardo.patimo@math.uni-freiburg.de}
}
}

\DeclareMathOperator{\cha}{ch}
\DeclareMathOperator{\chara}{char}
\DeclareMathOperator{\coeff}{coeff}
\DeclareMathOperator{\defect}{def}
\DeclareMathOperator{\Downs}{Downs}
\DeclareMathOperator{\End}{End}
\DeclareMathOperator{\Ext}{Ext}
\DeclareMathOperator{\for}{For}
\DeclareMathOperator{\gr}{gr}
\DeclareMathOperator{\Grad}{Gr}
\DeclareMathOperator{\grdim}{grdim}
\DeclareMathOperator{\grrk}{grrk}
\DeclareMathOperator{\hgt}{ht}
\DeclareMathOperator{\Hom}{Hom}
\DeclareMathOperator{\Id}{Id}
\DeclareMathOperator{\Iden}{Id}
\DeclareMathOperator{\ima}{Im}
\DeclareMathOperator{\Ker}{Ker}
\DeclareMathOperator{\rad}{rad}
\DeclareMathOperator{\rank}{rk}
\DeclareMathOperator{\Rep}{Rep}
\DeclareMathOperator{\res}{res}
\DeclareMathOperator{\sgn}{sgn}
\DeclareMathOperator{\spa}{span}
\DeclareMathOperator{\supp}{supp}
\DeclareMathOperator{\Sym}{Sym}
\DeclareMathOperator{\Trace}{Tr}
\DeclareMathOperator{\Peaks}{Peaks}
\DeclareMathOperator{\adm}{Adm}
\DeclareMathOperator{\Conf}{Conf}
\DeclareMathOperator{\rex}{rex}
\DeclareMathOperator{\cone}{cone}

\newcommand{\vertcong}{\mathbin{\rotatebox[origin=c]{-90}{$\cong$}}}

\definecolor{Mred}{RGB}{236,56,36}
\definecolor{Mgreen}{RGB}{72,220,38}
\definecolor{Mblue}{RGB}{45,30,200}
\definecolor{Darkgrey}{RGB}{74,74,74}
\definecolor{Mediumgrey}{RGB}{128,128,128}
\definecolor{Lightgrey}{RGB}{155,155,155}
\definecolor{Mbrown}{RGB}{139,87,42}
\definecolor{Mpurple}{RGB}{189,15,224}

\DeclareMathOperator{\Gr}{Gr}
\title{On the affine Hecke category for $SL_3$}

\author{Nicolas Libedinsky and Leonardo Patimo}

\begin{document}

\maketitle

\begin{abstract} 
We study the diagrammatic Hecke category associated with the affine Weyl group of type $\tilde{A}_2 $. More precisely we find a (surprisingly simple) basis for the Hom spaces between indecomposable objects,  that we call \emph{indecomposable double leaves.}

\end{abstract}

\section*{Introduction}

By the affine Hecke category we mean the diagrammatic Hecke category associated with an affine Weyl group. These categories are of crucial importance in representation theory and in the theory of Kazhdan-Lusztig combinatorics.  The main question regarding these categories is to understand the indecomposable objects, and the finest way to achieve this is to study the morphisms between them.  We therefore consider the following problem.

\begin{enumerate}[Problem A:\quad]
	\item[\textbf{Problem A: }] Understand the Hom spaces between indecomposable objects in the Hecke category. 
\end{enumerate}

Hom spaces contain a lot of critical categorical information.
For example, if one knows Hom spaces, one can compute intersection forms and understand the splitting behavior in the affine Hecke category in positive characteristic.

The solution to \textbf{Problem A} for $SL_2$ in characteristic zero is implicit in \cite{E3} (for the experts, the indecomposable light leaves are all the ``up'' morphisms). 

In this paper we study and solve \textbf{Problem A} for the group $SL_3$ in characteristic zero. 
We define the \emph{indecomposable double leaves},  a subset of the usual set of double leaves.  The main theorem of this paper  is \Cref{Thm3}: it states that  if $B_x$ and $B_y$ are indecomposable objects, then  the indecomposable double leaves are a basis for $\mathrm{Hom}(B_x,B_y)$.

We find noteworthy that the set of indecomposable light leaves for $SL_3$ has such a nice and concise form (see \eqref{ILLwall}-\eqref{ILLbeyond}). This is rather surprising since there are many non-canonical choices involved in the construction of the light leaves (unlike the case of $SL_2$).

A fundamental combinatorial problem about Kazhdan-Lusztig polynomials is to find a set that expresses the coefficients of Kazhdan-Lusztig polynomials. For affine Weyl groups in type $A$, this problem had only been solved \cite{LS2} (with completely different methods) for those elements which are maximal with respect to the finite Weyl group, where the combinatorics of Young tableaux is at disposal.
We remark that a consequence of our result is that the set of indecomposable light leaves provide a solution to this counting problem in type $\tilde{A}_2$ for the whole group.

An analogous result in the Grassmannian case (i.e., for maximally singular elements in type $A_n$) was obtained by the second author in \cite{Pat3}.
Moreover, in \cite{LW2}, the first author and Williamson describe a general framework in which similar results can be pursued for an arbitrary Coxeter group.
We believe it is worth trying to extend the methods of this paper to other types.

It is implicit that to solve to \textbf{Problem A} and construct the set of indecomposable light leaves, one first needs to understand well and be able to compute the indecomposable objects. For this reason, before we can approach \textbf{Problem A} we address the following problem.
\begin{enumerate}[Problem B:\quad ]
	\item[\textbf{Problem B: }]\label{QuestB} Find an explicit construction of the indecomposable objects. 
\end{enumerate}
For $SL_2$, the solution to  \textbf{Problem B}
 in characteristic zero was given by Elias in \cite{E3} and in positive characteristic by Burrull, the first author and Sentinelli in \cite{BLS}. 

To answer \textbf{Problem B} we see each indecomposable object as a  subobject of some Bott-Samelson.  We find explicit projectors (\Cref{Thm2}) from a Bott-Samelson to the top indecomposable summand.  This result is interesting in itself,  because it could pave the road (and at the very least it is a necessary first step) in order to find  the analogous projectors for $SL_3$ in positive characteristic.  This would have immense applications in modular representation theory since indecomposable objects in (a quotient of) the Hecke category correspond to indecomposable tilting modules (cf. \cite{RW}). In particular, computing projectors in positive characteristic would help us to understand tilting modules (and to possibly prove the ``billiards conjecture''
by Lusztig and Williamson \cite{LuWi} about characters of tilting modules for $SL_3$).

In our proof of the main result (\Cref{Thm3}), a main necessary ingredient is the knowledge of the graded rank of $\mathrm{Hom}(B_x,B_y)$. For this reason, we first had to obtain a closed formula for Kazhdan-Lusztig polynomials in the $SL_3$ case (\Cref{Thm1}).  This theorem is also interesting in itself for several reasons.  Before the present result, the only infinite groups where all Kazhdan-Lusztig polynomials could be explicitly computed (a result due to Dyer \cite{Dy2}) were the Universal Coxeter groups, i.e., those for which the Coxeter matrix has only $\infty$ in its off-diagonal entries.  Universal Coxeter groups are much simpler that the ones considered in this paper (as an example, each element has a unique reduced expression while in this work different reduced expressions play an important role). 

\Cref{Thm1} was a starting point for  the authors and Plaza to define the ``pre-canonical bases'' \cite{LPP}.  On the other hand Theorem 1  inspired the same calculation  for $\tilde{B}_2$ by Batistelli, Birgham and Plaza \cite{BBP}.  Finally,  in a work in progress by the authors and Plaza, this result is generalized  to any $\tilde{A}_n$ if the elements belong to the lowest double cell (in some precise sense, this is most of the affine Weyl group).

\subsection*{Kazhdan-Lusztig polynomials} 

Let $s_1,s_2$ and $s_3$ be the simple reflections in the affine Weyl group $W$ corresponding to $\tilde{A}_2$. For simplicity of notation we will often denote them simply by $1$, $2$ and $3$. We will use ``label mod $3$'' so $145$ will mean $s_1s_1s_2.$ 

Consider the element $x_n:=12345\cdots n$. For $m,n\in \mathbb{N}$, define $r:=s_0$, $s=s_{2m-2n}$, \begin{equation}\label{rextheta}
\theta(m,n):=1234\cdots (2m+1)(2m+2)(2m+1)\cdots(2m-2n+1)
\end{equation}
 and the set \[\tilde{\theta}(m,n)=\{ \theta(m,n), r\theta(m,n), \theta(m,n)s, r\theta(m,n)s\}.\]

Notice that $x_n$ has a unique reduced expression. We denote by $\undtheta(m,n)$ the reduced expression for $\theta(m,n)$ as in \eqref{rextheta}.

The symmetric group $\calS_3$ acts as group automorphisms on $W$ by permuting the indexes. 
For $x,y\in W$ we say $x\sim y$ if there exists $\sigma\in \calS_3$ such that $\sigma(x)=y$.

For every element of $x\in W$ there exists $y\in W$ with $x\sim y$ such that $y$ belongs to the disjoint union \[\{x_n\}_{n\in \mathbb{N}} \dot\cup\dot{\bigcup_{n,m\in \mathbb{N}}} \tilde{\theta}(m,n).\] We say that elements of the form $x_n$ are \emph{on the wall} and that elements in some $\tilde{\theta}(m,n)$ are \emph{beyond the walls}. 

Let $H$ be the Hecke algebra of $W$ with standard basis $\{\bfH_x\}$ and Kazhdan-Lusztig basis (or canonical basis) $\{\undH_x\}$ (cf. \cite{S4}). For $x\in W$, define the element $\bfN_x:=\sum_{y\leq x}v^{\ell(x)-\ell(y)}\bfH_y \in H.$

\begin{introthm}\label{Thm1}
\begin{enumerate}[i)]
\item The canonical basis on the wall is given by
\[ \undH_{x_n}=\bfN_{x_n}\;\text{ for }n\leq 3,\quad \undH_{x_4}=\bfN_{x_4}+v\bfH_{x_1}\]
and for  $n\geq 5$ by	\[\undH_{x_n}=\begin{cases}
	 \bfN_{x_n}+v\bfN_{x_{n-3}} & \text{ if $n$ odd}\\
	 \bfN_{x_n}+v\bfN_{x_{n-3}} + v\bfH_{z_n} +v^2 \bfH_{z'_n} & \text{ if $n$ even}
	 \end{cases}\] where $z_n=13456\ldots (n-2)$ and $z'_n= z_n s_n$.
\item 	 The canonical basis beyond the wall is given by: \[\undH_{\theta(m,n)} =\sum_{i=0}^{\min(m,n)}v^{2i}\bfN_{\theta(m-i,n-i)}\]
 \[\undH_{r\theta(m,n)}=\undH_{r}\undH_{\theta(m,n)}, \ \ \ \ \undH_{\theta(m,n)s}=\undH_{\theta(m,n)}\undH_s,\]
 \[\undH_{r\theta(m,n)s}=\undH_{r}\undH_{\theta(m,n)}\undH_s.\]
	 \end{enumerate}
\end{introthm}

This result (without a proof) appears in an unpublished paper by Geordie Williamson, that he kindly shared with us. Progress in this direction had also been done by Wang \cite{Wang}.

This theorem gives an explicit description of the Kazhdan-Lusztig basis, but not explicit enough for our purposes. We also need to understand the element $\bfN_x$ for any $x$, or, equivalently, to understand the set $\leq x:=\{y\leq x\ \vert \ y \in W\}$. We describe now such a set for the elements $\theta(m,n)$. 

For clarity of exposition we start by explaining the figure below. The big black dot represents the identity in $W$. The colors on the edges represent the following simple reflections: red is $s_1$, green is $s_2$ and blue is $s_3$. Concerning the little yellow dots, the lower one is $\theta(2,0)$ the middle one is $\theta(3,1)$ and the top one is $\theta(4,2)$.

In the figure there are three shades of gray, let us call them \emph{light grey}, \emph{middle gray}, and \emph{dark gray}. The light grey is the set $\leq \theta(2,0)$. This region is an ``equilateral triangle'' defined by the following three properties. Its center $O$ is the upper vertex of the identity triangle, its three medians are the three reflecting hyperplanes passing through $O$ and it is the minimal triangle satisfying the first two conditions and containing $ \theta(2,0)$.

The middle gray joined with the light grey region is the set $\leq \theta(3,1)$. The construction of this is obtained by adding all the ``red hexagons'' surrounding $\leq \theta(2,0)$. Finally the union of the light, middle and dark gray regions is the set $\leq \theta(4,2)$, which is obtained as before by adding all the red hexagons surrounding $\leq \theta(3,1)$.

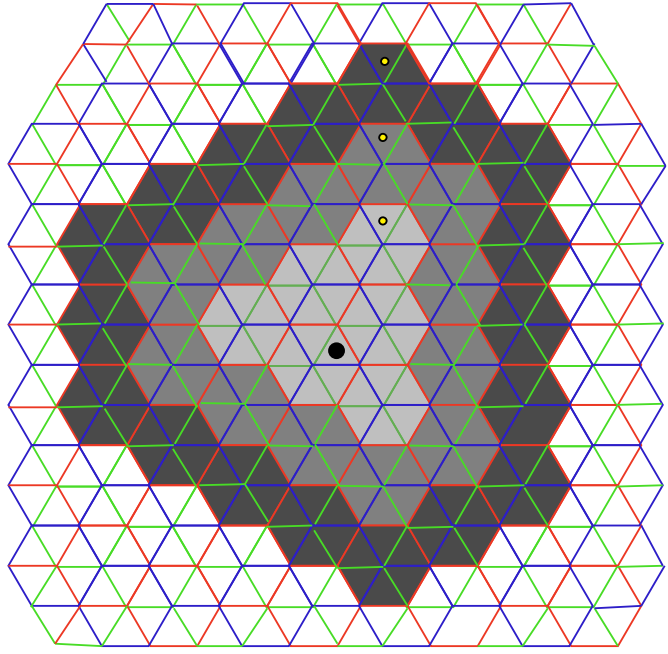
\begin{center}

\tikzset{every picture/.style={line width=0.7pt}}


 \end{center}

What we have just described for this particular case is a general construction for \newline
$\leq\theta(m,n)$: the set $\leq\theta(i,0)$ (or $\leq\theta(0,i)$) is the equilateral triangle satisfying the corresponding three properties mentioned before and $\leq\theta(i+1,j+1)$ is obtained from $\leq\theta(i,j)$ by adding the corresponding surrounding hexagons. The description of $\bfN_{x_n}$ is also easy to describe geometrically as it is shown in \Cref{counting}.

\subsection*{Projectors in Bott-Samelson objects}
Consider the diagrammatic Hecke category $\mathcal{H}$ defined by Elias and Williamson \cite{EW2}. Our second theorem gives an explicit inductive procedure to find an idempotent $e_{\underline{w}}$ in the endomorphism ring of some Bott-Samelson object $\mathrm{BS}(\underline{w})$ projecting to the corresponding indecomposable object $B_w$. 

We introduce some notation. A rectangle with $w$ inside represents the projector $e_{\underline{w}}$. Furthermore, we replace $x_n$ simply by $n$ in that notation (so a rectangle with an $n$ inside represents the projector $e_{x_n}$) and a rectangle with $m,n$ inside represents a projector $e_{\theta(m,n)}$ composed with some braid move. One can check that the projector formula does not depend on the chosen braid move.

Finally, for simplicity of some diagrams, we will use the following notation:

\vspace{-2mm}
\tikzset{every picture/.style={line width=2pt}}
\begin{center}

 \end{center}

\vspace{-2mm}
 
\begin{introthm}\label{Thm2}
\begin{itemize}
\item The projectors on the wall are described by the following recursive formulas:

\vspace{-2mm}

%

\vspace{-3mm}
\end{center}
where $y_n$ is the reduced expression $123\cdots (n-4)(n-2)(n-3)(n-2)$\footnote{For $n\leq 3$ the elements $y_n$ and $z_n$ are not well-defined. However, for those $n$ the corresponding coefficients $c_{n+1}$ and $d_{n+1}$ in the projectors vanish.} and 
\begin{itemize}
\item $c_1=c_2=0$ and for $n\geq 1$, $c_{2n+1}=c_{2n+2}=-\frac{n-1}{n}$
\item $d_1=d_3=d_5=0$ and for $n\geq 2$, $d_{2n+1}=\frac{n-2}{n-1}$.
\end{itemize}

\item The projectors beyond the wall are described by the following formula:
\vspace{-3mm}
\end{itemize}
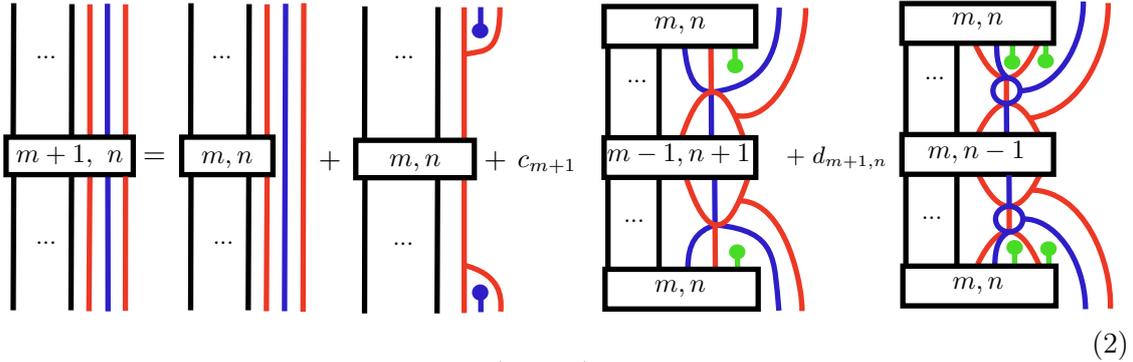
\begin{equation}\label{projoow}

\end{equation}
where $c_{m+1}=\frac{m}{m+1}$ and $d_{m+1,n}=-\ \frac{n( m+n+1)}{( n+1)( m+n+2)}$ for every $m,n\geq 0$.

\end{introthm}

As we mentioned before, the analogue result for $\tilde{A}_1$ in characteristic zero was given by Elias \cite{E3} (it is much simpler that the present result). For Universal Coxeter groups in characteristic zero it was obtained by the Elias and the first author \cite{EL2} and for $\tilde{A}_1$ in positive characteristic by Burrull, the first author and Sentinelli \cite{BLS}. The ``beyond the wall'' part of our result is strongly inspired by a combination of Elias's triple clasp expansion \cite{E4} in the context of quantum groups, Elias's quantum Satake equivalence \cite{E3} and Williamson's singular Soergel bimodules \cite{W4}. Although with these tools one gets projectors using singular diagrams, not regular diagrams as the ones given here. As we said before, we hope that this formula opens the way for a formula in $\tilde{A}_n$ by understanding better which kind of light leaves appear in the projectors.
 
\subsection*{Categorifying the KL polynomials}

Recall from \cite{Li1,Li3} the definition of the ``Light leaves basis'' and from \cite{EW2} its diagrammatic formulation as a sequence of elements in the set $\{U0,U1, D0, D1\}$. We call a light leaf that only contains $U0$'s and $U1$'s a \emph{$U$-light leaf}. A box with a letter $U$ inside will mean a $U$-leaf.


For $\mathrm{BS}(\underline{x_n})$ and $\mathrm{BS}(\undtheta(m,n))$ we define a set of \emph{indecomposable light leaves} $\mathrm{ILL}(x_n)$ and $\mathrm{ILL}(\theta(m,n))$ as follows. 

The set $\mathrm{ILL}(x_n)$ is composed by the following set of maps precomposed with $e_{x_n}$ :

\begin{center}
\begin{equation}\label{ILLwall}

\end{equation}
\end{center}

 The set $\mathrm{ILL}(\theta(m,n))$ is composed by the following set of maps precomposed with $e_{\undtheta(m,n)}$:

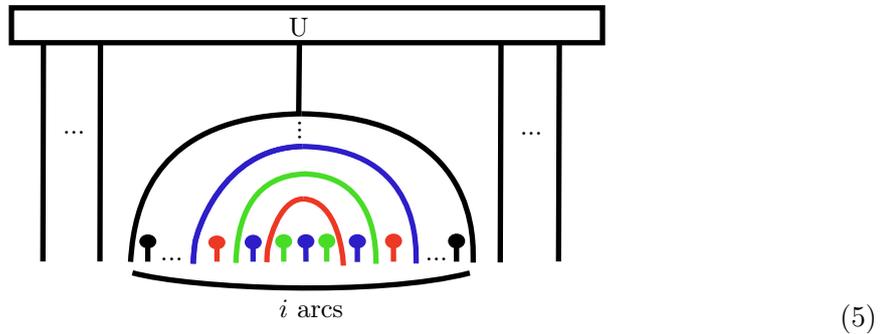
\begin{center}
\begin{equation}\label{ILLbeyond}
%
\end{equation}

\end{center} 
where $1\leq i\leq \mathrm{min}\{m,n\}.$ In this last picture the colors depend of the residue of $2m$ modulo $3$. In this picture we are supposing that $s_{2m+2}=b$ (or, equivalently that $2m+2$ is a multiple of $3$). 

It is now clear how to produce $\mathrm{ILL}(x)$ for any $x\in W$ (it will be a set of maps with source $\mathrm{BS}(\underline{x})$). If the element $x$ can be obtained from some $x_n$ or from some $\theta(m,n)$ by permuting the colors, we just need to permute correspondingly the colors in the pictures above. If the element $x$ is $\theta(m,n)$ multiplied by $s$ on the left and/or by $t$ on the right, as in \Cref{Thm1}, then a simple procedure (explained in \Cref{inducinglightleaves})
produces the indecomposable light leaves $\mathrm{ILL}(x)$ starting from $\mathrm{ILL}(\theta(m,n))$.

We denote by $\bar{\ILL}(y)$ the set obtained by flipping upside-down the morphisms in $\ILL(y)$.
In the following theorem we see the indecomposable objects as subobjects of the corresponding Bott-Samelson objects via the projectors defined above.
\begin{introthm}[Indecomposable Double leaves theorem]\label{Thm3}
For $x,y\in W$ the set $ \bar{\ILL}(y)\circ \ILL(x)$ forms a free $R$-basis of $\Hom(B_x,B_y)$ in $\mathcal{H}$.
\end{introthm}

\subsection*{Organization of the paper} \Cref{S2} is devoted to prove \Cref{Thm1}. It can be read without knowing the Hecke category. \Cref{S3} and \Cref{S4} are resp. devoted to prove \Cref{Thm2} and \Cref{Thm3}. These two sections use the results in \Cref{S2} and are largely independent from each other except that at one point: in the proof of \Cref{deg0onthewall} we use the knowledge of $\ILL(x_n)$ from \Cref{S4}.

\section{The Hecke algebra (\texorpdfstring{\Cref{Thm1}}{Theorem 1})}\label{S2}
Let $(W,S)$ denote the Coxeter system of type $\tilde{A}_2$: it has three simple reflections $s_1,s_2,s_3$ and $m_{s_1,s_2}=m_{s_2,s_3}=m_{s_3,s_1}=3$. Let $\leq$ denote the Bruhat order on $W$, $\ell$ the length function and let us say that $y\lessdot x$ if and only if $y\leq x$ and $\ell(y)=\ell(x)-1$.

\subsection{Some results about \texorpdfstring{$W$}{W}}\label{Sgroup}

If $\undw=r_1r_2\cdots r_n$ is an expression (a sequence of simple reflections), we say that there is a \emph{braid triplet in position $i$} if $r_{i-1}=r_{i+1}$ and $r_{i-1}\neq r_i$ (here $1<i<n$). We define 
 the \emph{distance} between a braid triplet in position $i$ and a braid triplet in position $j>i$ to be the number $ i-j-1$. 

\begin{lemma}\label{redex}
	An expression $\undw$ without adjacent simple reflections is reduced if and only if the distance between any two braid triplets is odd. 
\end{lemma}
\begin{proof}

We recall that expressions are in bijection with paths in the Coxeter complex (in the case of $\tilde{A}_2$ the Coxeter complex is the usual tessellation of the plane by equilateral triangles) starting from the identity. Under this bijection, reduced expressions are geodesics.

In Figures $1$ and $2$, the triangle with the black dot denotes the identity and the red arrows are subexpressions that do not contain any braid move. When an expression is given as a concatenation of these arrows (or, equivalently, it is divided into maximal subexpressions without braid triplets) it is clear that a path is a geodesic if and only if each arrow has the same direction and sense as the arrow that is two steps ahead of it. Thus, the first and third arrows have the same direction, and so do the second and fourth, etc.

Let us consider all the red arrows that are not the last and the first one. Then, the number of triangles touched by each one of these arrows minus two (we have to subtract the starting and the ending triangles) is the distance between two braid triplets. It is clear that the sense of the next arrow depends on the parity of this distance, thus the Lemma is proved. 

We give an example. In Figure $1$, the distance between braids are $5, 3$ and $1$, and thus the path is a geodesic. In Figure $2$, the distances are $5, 3$ and $2$, thus the last arrow makes the path a non-geodesic.

\tikzset{every picture/.style={line width=0.7pt}}

\end{proof}

\begin{cor}\label{stu}
	Let $s,t,u$ be the three simple reflections of $W$. If $w\in W$ has a reduced expression $\undw$ ending in $st$, then $wu>w$.
\end{cor}
\begin{proof}
	Multiplying $\undw$ by $u$, we do not create any new braid triplets. Hence, $\undw u$ is reduced if and only if $\undw$ is reduced. 
\end{proof}

Recall from the introduction that $x_n=123123\ldots $ with length $n$ and \[\theta(m,n):=1234\cdots (2m+1)(2m+2)(2m+1)\cdots(2m-2n+1).\] 

\begin{cor}\label{lessdot}
	For every $x\in W$ there are at most $6$ elements $y$ with $y\lessdot x$.
\end{cor}
\begin{proof}
	We first assume $x=x_n$ for some $n$. Every element $y\lessdot x_n$ can be obtained by removing a simple reflection from $123\cdots n$. If we remove a simple reflection which is not one of the first two or of the last two, then $y$ would contain a subword of the form $stst$ and it is not reduced. On the other hand, removing one of the first two or of the last two gives a reduced word. So, if $n\geq 4$, there are exactly $4$ elements such that $y\lessdot x_n$.
	
	Consider now $x$ beyond the wall. 
	We can consider a reduced expression $\undx$ of $x$ with exactly one braid triplet. 
	From the discussion above, we see that the only simple reflections that might be removed from $\undx$ in order to obtain a reduced word are 
	the first two, the last two or the three simple reflections in the unique braid triplet. But the the middle simple reflection in the braid triplet does not give a reduced expression. Hence we end up with at most $6$ possibilities.	 
\end{proof}

From \Cref{lessdot} we see that there are at most $6$ elements $y$ such that $y\lessdot {\theta}(m,n)$, but we can be more precise. In the reduced expression $\undtheta(m,n)$ for $\theta(m,n)$ given above, the removal of the second simple reflection does not give a reduced expression because we end up with two braid triplets at even distance (cf. \Cref{redex}). By symmetry, the same happens if we remove the second to last simple reflection. So, if $m,n>0$ there are exactly $4$ elements with $y\lessdot \theta(m,n)$. If $m>0$ and $n=0$ there are $3$.

For $x\in W$ we denote by $|x|$ the number of elements smaller or equal than $x$ in the Bruhat order (in formulas, $|x|:=\vert \leq x \vert $). It is clear that $x\sim y \implies |x|=|y|$.

\begin{lemma}\label{counting}
For all $m,n\geq 0$ we have \[|\theta(m,n)|= 3m^2+3n^2+12mn+9m+9n+6.\]
For any $n\geq 1$ we have \[|x_{2n}| = 3n^2+n\ \ \text{and} \ \ |x_{2n+1}|= 3n^2+5n.\] 	
\end{lemma}

\begin{proof}
We will prove the first equality. First we have to prove that the picture after Theorem $1$ is correct. For this we will first prove that $\leq \theta(m,0)$ is the equilateral triangle (with zigzag sides instead of straight lines) described before that picture, i.e., the triangle whose center $O$ is the upper vertex of the identity triangle, its three medians are the three reflecting hyperplanes passing through $O$ and it is the minimal triangle satisfying the first two conditions and containing $ \theta(m,0)$. The proof of this is done by induction on $m$. Let us assume the hypothesis for $m$. Without loss of generality, suppose $m\equiv 2\, (\mathrm{mod}\, 3) $ (as in the picture, where $m=2$). We need to prove that $\leq \theta(m+1,0)$ is the smallest equilateral triangle composed by a union of blue hexagons (by this we mean, of course, hexagons with all edges colored blue) containing $\leq \theta(m,0)$. We call this triangle $BT$ (for blue triangle). 

It is clear that $2123123\cdots$ is another reduced expression for $\theta(m,0)$. Using this we see that ${\theta(m+1,0)}={\theta(m,0)}rg$ (here we use $r=s_1$ for the red simple reflection, $g=s_2$ for the green and $b=s_3$ for the blue). So \[\leq\theta(m+1,0)=\leq\theta(m,0) \cup (\leq\theta(m,0)) r \cup (\leq\theta(m,0)) g \cup (\leq\theta(m,0))rg.\] 

From this equation it is clear that $BT\subseteq \leq\theta(m+1,0).$ The inverse inclusion follows by noticing that if a triangle is inside a blue hexagon, just by multiplying by $r$ and $g$ one stays in the same blue hexagon. With this we conclude the proof of the fact that $\leq \theta(m,0)$ is the equilateral triangle described in the introduction. 

Without loss of generality, suppose that $\leq \theta(m,n)$ is the union of some set of red hexagons (i.e., $\vert m-n \vert \, \equiv 2\, (\mathrm{mod}\, 3)$, as in the figure after \Cref{Thm1}). We will prove that $\leq \theta(m+1,n+1)$ is the union of red hexagons obtained by adding to $\leq \theta(m,n)$ all the red hexagons that are adjacent to $\leq \theta(m,n)$. We call this set $RH$. We have that $\theta(m+1,n+1)=\theta(m,n)rgbg.$
Again we have \[\leq \theta(m+1,n+1)=\, (\leq \theta(m,n))(\leq rgbg).\]
As any element in $RH$ can be written as the product of an element in $\leq\theta(m,n)$ and an element in $\leq rgbg$, we see that $RH\subseteq \, \leq \theta(m+1,n+1).$ The inverse inclusion follows by noticing that $\leq \theta(m,n)r\subset RH$ and, as before, multiplying by $b$ and $g$ do not change the red hexagon in which an element is.

Now that we have given a geometric description of the set $\leq\theta(m,n)$ for all $m, n$, the counting formula of the lemma follows easily using the following observations. 

\begin{itemize}
\item $\vert \leq \theta(m,0) \vert =6(\sum_{i=1}^{m+1}i)$ by counting the number of hexagons in $ \leq \theta(m,0).$
\item $\vert \leq\theta(m+1,n+1)\vert= \vert \leq\theta(m,n)\vert+18(m+2)+18n$ by counting the number of hexagons one needs to add to $\leq\theta(m,n)$ to obtain $\leq\theta(m+1,n+1).$
\end{itemize}

The second equality follows similar lines and we will not prove it in detail. We just leave the reader with the picture of the sets $\leq x_4, \leq x_6, \leq x_8$ (resp. in light gray, middle gray and dark gray): 

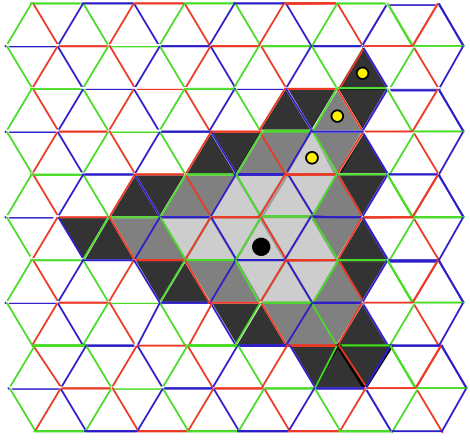
\begin{center}

\tikzset{every picture/.style={line width=0.7pt}}	
	
\end{center}\vspace{-7mm}
\end{proof}

\subsection{Some notations for the Hecke algebra}

Let $H$ be the Hecke algebra of $W$. It is a free $\bbZ[v,v^{-1}]$-module with two distinguished bases: the standard basis $\{\bfH_x\}$ and Kazhdan-Lusztig basis $\{\undH_x\}$ (see \cite{S2}). We have $\undH_x=\sum_{y\leq x}h_{y,x}(v)\bfH_y$, with $h_{x,x}(v)=1$ and $
h_{y,x}(v)\in v\bbN[v]$. The polynomials $h_{y,x}$ are the Kazhdan-Lusztig polynomials.
For $x\in W$, define the element $$\bfN_x:=\sum_{y\leq x}v^{\ell(x)-\ell(y)}\bfH_y \in H.$$

If $y\leq x$ the $\mu$-coefficient $\mu(y,x)$ is defined as the coefficient of $v$ in $h_{y,x}(v)$. For $x\in W$ and $s\in S$ such that $xs>x$ we have
\begin{equation}\label{muinduction}
 \undH_x\undH_s = \undH_{xs}+\sum_{\substack{y <x\\ ys<y}}\mu(y,x) \undH_y.
\end{equation}

If $X=\sum_{w\in W} p_{w}(v)\bfH_w \in H$, we define the \emph{content} of $X$ to be $c(X)=\sum_{w\in W} p_w(1)\in \bbZ[v,v^{-1}]$. For example, we have $c(\bfN_w)=|w|$ for every $w\in W$. The content $c:H\raw \bbZ$ is a ring homomorphism, in particular it satisfies $c(X\undH_s)=2c(X)$ for any $X\in H$ and $s\in S.$
Notice that for every $x,y\in W$ with $x\sim y$ we have $c(\undH_x)=c(\undH_y)$. Moreover, since $c$ commutes with the anti-involution of $H$ fixing all $\bfH_s$ with $s\in S$, we have $c(\undH_x)=c(\undH_{x^{-1}})$ for every $x \in W$.

If $X,Y\in H$, we say that $X\geqH Y$ if $X-Y=\sum_{w\in W}p_w(v) \bfH_w$, with $p_w(v)\in \bbN[v,v^{-1}]$ for every $w\in W$. 
Notice that $X\geqH Y$ and $c(X)=c(Y)$ implies that $X=Y$.

Let $X\in H$. We say that $X$ is \emph{perverse} if $X$ is a positive combination of elements in the KL basis, i.e., if
\[X = \sum_{w\in W}n_w \undH_w\qquad \text{with }n_w\in\bbN.\]
If $X$ is perverse, we say that $\undH_w$ is a \emph{summand} of $X$ if $\undH_w$ occurs in $X$ with non-trivial coefficient. We recall that KL polynomials satisfy the following monotonicity property (see \cite{BM,Pla}).
\begin{theorem}[Monotonicity of KL polynomials]\label{monothm}
	Let $x,y,w\in W$ with $x\leq y\leq w$. Then we have
	\[ h_{y,w}(v)-v^{\ell(y)-\ell(x)} h_{x,w}(v)\in \bbZ_{\geq 0}[v].\]
\end{theorem}

As a corollary, if $h_{x,w}(v)=v^{\ell(w)-\ell(x)}+\sum_{i=1}^{\ell(w)-\ell(x)-2} c_iv^i$, then 
\begin{equation}\label{Ngeq}
\undH_w\geqH \bfN_w+\sum^{\ell(w)-\ell(x)-2}_{i=1} c_iv^i\bfN_{x}.
\end{equation}

\subsection{Canonical basis on the wall}

Recall that $y_n=123\ldots (n-2)n$ (this is a different reduced expression than the one used in the introduction), $z_n=13456\ldots (n-2)$ and $z'_n$ is $z_n$ without the last simple reflection (i.e., $z_n'=z_ns_n$). Let $s$ be the unique simple reflection that is not in the left descent set of $\theta(n-2,0)$. Notice that $y_{2n}= \theta(n-2,0)$, $y_{2n+1} \sim s\theta(n-2,0)$ and $z_{2n}^{-1}\sim \theta(n-3,0)$. 

We can compute via direct computation the KL basis for small elements. 
We have $\undH_{x_n}=\bfN_{x_n}$ for all $n\leq 3$ and $\undH_{x_4}=\bfN_{x_4}+v\bfN_{x_1}$.
The general formula is given in the following proposition.
We remark that statement $(A_n)$ is \Cref{Thm1}.i). 

\begin{prop}\label{Heckeonthewall}
	For every $n\geq 5$ the following two statement hold:
\begin{enumerate}	
	
		 \item[$(A_n)$:] 	\[\undH_{x_n}=\begin{cases}
	\bfN_{x_n}+v\bfN_{x_{n-3}} & \text{ if $n$ odd}\\
	\bfN_{x_n}+v\bfN_{x_{n-3}} + v\bfH_{z_n} +v^2 \bfH_{z'_n} & \text{ if $n$ even}
	\end{cases}\]
	\item[$(B_n)$:] Let $s=s_{n+1}$.
	\[ \undH_{x_n}\undH_s=\begin{cases}
			\undH_{x_{n+1}}+\undH_{y_n} &\text{ if $n$ odd} \\
		\undH_{x_{n+1}}+\undH_{y_{n}}+ \undH_{z_{n}}&\text{ if }n\text{ even}.
		\end{cases}\] 

\end{enumerate}
\end{prop} 

\begin{proof}
	We prove the two statements in the order $(A_n)\implies (B_n)\implies (A_{n+1})$. The statement $(A_5)$ can be easily checked directly.
	
	We first show $(A_n)\implies (B_n)$. From $(A_n)$ it follows that $\mu(y,x_n)=0$ unless $y=x_{n-3}$, $y\lessdot x_n$ or  $n$ is even and $y=z_n$. 
	Notice that $x_{n-3}s>x_{n-3}$ and $z_{n}s<z_n$.
	
	By the proof of \Cref{lessdot}, there are exactly four elements of $y$ such that $y\lessdot x_n$. Among these, the only one satisfying $ys<s$ is $y=y_n$.
	We can now apply equation \eqref{muinduction} and obtain $(B_n)$.
	
	We assume now $(A_n)$ and $(B_n)$ and show $(A_{n+1})$.
	We can assume that $n$ is even, the case with $n$ being odd is similar. Let $n=2m$.
	From $(A_n)$ it follows that $c(\undH_{x_{2m}})=|x_{2m}|+| x_{2m-3}|+2$ 
	and from $(B_n)$ we have
	\begin{eqnarray}
	c(\undH_{x_{2m+1}}) & = & c(\undH_{x_{2m}}\undH_{s})-c(\undH_{y_{2m}})-c(\undH_{z_{2m}}) \nonumber\\
	& = & 2(|x_{2m}|+|x_{2m-3}|+2)-| \theta(m-2,0)|-|\theta(m-3,0)|\nonumber \\
	& = & 6m^2+2\nonumber \\
	& = & |x_{2m+1}|+|x_{2m-2}| = c(\bfN_{x_{2m+1}}+v\bfN_{x_{2m-2}}).\nonumber
	\end{eqnarray}	
	
	Hence, to show $(A_{n+1})$ it is enough to prove $\undH_{x_{2m+1}}\geqH \bfN_{x_{2m+1}}+v\bfN_{x_{2m-2}}$. From $(A_n)$ and a direct computation we obtain that the coefficient of $\bfH_{x_{2m-2}}$ in $\undH_{x_{2m}}\undH_{s_{2m+1}}$ is $v^3+2v$.
	From $(B_n)$ we obtain 
	\[h_{x_{2m-2},x_{2m+1}}(v)=v^3+2v - h_{x_{2m-2},z_{2m}}(v) -h_{x_{2m-2},y_{2m}}(v).\]
	But $h_{x_{2m-2},z_{2m}}(v)=0$ and $h_{x_{2m-2},y_{2m}}(v)=v$ because $x_{2m-2}\not\leq z_{2m}$ and $x_{2m-2}\lessdot y_{2m}$. So we obtain $h_{x_{2m-2},x_{2m+1}}(v)=v^3+v$. Therefore, by equation \eqref{Ngeq}, we get
	$$\undH_{x_{2m+1}}\geqH \bfN_{x_{2m+1}}+v\bfN_{x_{2m-2}}$$ as desired.
\end{proof}

\subsection{Canonical basis beyond the Wall}
The element $\theta(m,n)=123\ldots (2m+1)(2m+2)(2m+1)2m \ldots (2m+1-2n)$ has length $\ell(\theta(m,n))=2m+2n+3$.

\begin{lemma}\label{A_0}
	We have $\undH_{\theta(m,0)} = \bfN_{\theta(m,0)}$.
\end{lemma}
\begin{proof}
We prove it by induction on $m$. The case $m=0$ is easy. 
Let $t=s_{2m+1}$, so that $x_{2m+2}t = \theta(m,0)$. By \Cref{Heckeonthewall} we see that $\mu(y,x_{2m+2})$ is non-zero only if $y$ is $x_{2m-1}$, $z_{2m+2}$ or $y\lessdot x_{2m+2}$. The first two do not have $t$ in their right descent set. By \Cref{lessdot} we have 
\[\undH_{x_{2m+2}}\undH_t= \undH_{\theta(m,0)}+\undH_{x_{2m+1}}+\undH_{z_{2m+4}}.\] 

By \Cref{Heckeonthewall} we have \[c(\undH_{x_{2m+2}}\undH_t)=2(| x_{2m+2}|+| x_{2m-1}|+2),$$ $$c(\undH_{x_{2m+1}})= | x_{2m+1}|+| x_{2m-2}|.\] We know that $z_{2m+4}^{-1}\sim \theta(m-1,0),$ 
so we have $c(\undH_{z_{2m+4}})= | \theta(m-1,0)|.$
Using \Cref{counting} we obtain \[c(\undH_{\theta(m,0)})=
3m^2+9m+6.\] We also have $ c(\bfN_{\theta(m,0)})=3m^2+9m+6.$ Since we have $\undH_{\theta(m,0)}\geqH \bfN_{\theta(m,0)}$, we finally obtain $\undH_{\theta(m,0)}= \bfN_{\theta(m,0)}.$
\end{proof}

\begin{prop}\label{out}
	\label{goingout} The following three formulas hold for every $m$ and $n$ (where the terms involving $m-1$ and $n-1$ are neglected if $m=0$ or $n=0$).
	\begin{enumerate}
			\item[$(A_{m,n})$:]	\[\undH_{\theta(m,n)} =\sum_{i=0}^{\min(m,n)}v^{2i}\bfN_{\theta(m-i,n-i)}\]

		\item[$(B_{m,n})$:] For $r:=s_0$, $s=s_{2m-2n},$
		\begin{center}$\undH_{\theta(m,n)}\undH_s=\undH_{\theta(m,n)s},\ \ $ $\undH_{r}\undH_{\theta(m,n)}=\undH_{r\theta(m,n)}$ and \end{center} $$\undH_{r}\undH_{\theta(m,n)}\undH_s=\undH_{r\theta(m,n)s}$$
			\item[$(C_{m,n})$:] For $t:=s_{2m-2n-1},$	\[ \undH_{\theta(m,n)}\undH_s\undH_t = \undH_{\theta(m,n+1)}+\undH_{\theta(m,n)}+\undH_{\theta(m+1,n-1)}+\undH_{\theta(m-1,n)}\]
	\end{enumerate}
\end{prop}
\begin{proof}

The proof is by double induction on $m$ and $n$. For $m=n=0$ it is easy.
Fix $m,n$ and assume that the three statements hold for all pairs $(m',n')$ with $m'+n'<m+n$.
First we will prove $(A_{m,n})$, then $(B_{m,n})$ and finally. $(C_{m,n})$

\vspace{0.1cm}

\emph{\textbf{Step 1}: Proof of $(A_{m,n})$.} 

\vspace{0.1cm}

If $m',n'$ are such that $m'+n'<m+n$, then using $(A_{m',n'})$, \Cref{counting} and a straightforward computation we obtain
\begin{equation}\label{cundH}
c(\undH_{\theta(m',n')})=3m'n'(m'+n')+3(m'+n')^2+6m'n'+9(m'+n')+6.
\end{equation}
Using $(C_{m,n-1})$ one can check that \eqref{cundH} holds for $\theta(m,n)$ as well or, equivalently, that
\begin{equation}\label{newtheta} c(\undH_{\theta(m,n)})= \sum_{i=0}^{\min(m,n)}c(\bfN_{\theta(m-i,n-i)}).
\end{equation}

To conclude it is enough to check that for $1\leq i\leq \mathrm{min}(m, n)$, the monomial $v^{2i}$ occurs in $h_{\theta(m-i,n-i),\theta(m,n)}(v)$. Indeed, by monotonicity (\Cref{monothm}), this would imply that for any $i>0$ we have
$\undH_{\theta(m,n)}\overset{H}{\geq} v^{2i}\bfN_{\theta(m-i,n-i)}$. Since terms corresponding to different $i$ lie in different degrees we obtain

 \[\undH_{\theta(m,n)}\overset{H}{\geq} \sum_{i=0}^{\min(m,n)}v^{2i}\bfN_{\theta(m-i,n-i)}\]
and it must be an equality by \eqref{newtheta}.

By $(A_{m',n'})$, for any $m',n'$ with $m'+n'<m+n$ we have
\begin{equation}
\label{step3eq1}\undH_{\theta(m',n')}=\bfN_{\theta(m',n')}+v^2\undH_{\theta(m'-1,n'-1)}.\end{equation}
By $(C_{m-1,n-2})$ we have
\begin{equation}\label{step3eq2}\undH_{\theta(m-1,n-2)}\undH_{st}=\undH_{\theta(m-1,n-1)}+\undH_{\theta(m-1,n-2)}+\undH_{\theta(m,n-3)}+\undH_{\theta(m-2,n-2)}.
\end{equation}
If we rewrite $(C_{m,n-1})$ (remark that $s$ and $t$ in that equation agree with those in equation \eqref{step3eq2}) by plugging equation \eqref{step3eq1} in every term, and then using equation \eqref{step3eq2} to simplify, we obtain

%
\begin{equation}\label{nnnn} \undH_{\theta(m,n)}= v^2 \undH_{\theta(m-1,n-1)}+ \left(\bfN_{\theta(m,n-1)} \undH_{st}-\bfN_{\theta(m,n-1)}-\bfN_{\theta(m+1,n-2)}-\bfN_{\theta(m-1,n-1)}\right).\end{equation}

We claim that for any $i>0$ the coefficient of the term $v^{2i}\bfH_{\theta(m-i,n-i)}$ in term in parenthesis in \eqref{nnnn} is zero. This can be checked directly via an easy local computation if $i=1$ or $i=2$.

Recall that $\bfH_x\undH_s=\bfH_{xs}+v\bfH_{x}$ if $xs>x$ and $\bfH_{xs}+v^{-1}\bfH_{x}$ if $xs<x$. It follows that in $\bfN_x\undH_{st}$ only terms $v^i\bfH_y$ with $i\geq \ell(x)-\ell(y)-2$ can occur.
In particular, if $i>2$
the term $v^{2i}\bfH_{\theta(m-i,n-i)}$ cannot occur in $\bfN_{\theta(m,n-1)} \undH_{st}$. By length reasons, $v^{2i}\bfH_{\theta(m-i,n-i)}$ can neither occur in the other terms in the parenthesis on the right hand side of the equation \eqref{nnnn}. 

It follows that for any $i>0$ the coefficient of the term $v^{2i}\bfH_{\theta(m-i,n-i)}$ in $\undH_{\theta(m,n)}$ coincides with the coefficient of $v^{2i-2}\bfH_{\theta(m-i,n-i)}$ in $\undH_{\theta(m-1,n-1)}$, which is $1$ by $A_{m-1,n-1}$.

\vspace{0.1cm}
	
	\emph{\textbf{Step 2}: Proof of $(B_{m,n})$.}
	
	\vspace{0.1cm}
	
		 From $(A_{m,n})$ it follows that $\mu(y,\theta(m,n))=0$ unless $y\lessdot\theta(m,n)$. Using \Cref{lessdot}, however, it is easy to check that for all these $y$ we have $ys>y$ and $ry>y$.
		This immediately implies that $\undH_{\theta(m,n)}\undH_s=\undH_{\theta(m,n)s}$ and $\undH_{r}\undH_{\theta(m,n)}=\undH_{r\theta(m,n)}$.
		
		Consider now $\undH_{r}\undH_{\theta(m,n)}\undH_s$. Observe that for $y\lessdot \theta(m,n)$ we have $\ell(rys)=\ell(y)+2$ (cf. \Cref{lessdot}). This implies that the basis element $\undH_y$ (with $y<r\theta(m,n)s$) is a summand in $\undH_{r}\undH_{\theta(m,n)}\undH_s$ if and only if $v^2$ occurs in $h_{y,\theta(m,n)}(v)$ and $\ell(rys)=\ell(y)-2$. But $v^2$ can occur only if $y=\theta(m-1,n-1)$ or $\ell(y)=\ell(\theta(m,n))-2$. However, one can see (using \Cref{lessdot}) that in all these cases the condition $\ell(rys)=\ell(y)-2$ is never satisfied.
	
	\vspace{0.1cm}
	\emph{\textbf{Step 3}: Proof of $(C_{m,n})$. }
	\vspace{0.1cm}

	Since by $(B_{m,n})$ we have $\undH_{\theta(m,n)}\undH_s=\undH_{\theta(m,n)s}$, an element $y$ is such that $\undH_y$ is a summand of $\undH_{\theta(m,n)}\undH_s\undH_t$ if and only if it satisfies one of the following conditions:	
	\begin{enumerate}
		\item $y=\theta(m,n+1)$ or $y=\theta(m,n)$ (this last one because $\theta(m,n)t<\theta(m,n).)$
		\item $v$ occurs in $h_{ys,\theta(m,n)}(v)$, and $yt<y>ys$.
		\item $v^2$ occurs in $h_{y,\theta(m,n)}(v)$, and $yst<ys<y$.
	\end{enumerate}

Assume $y$ is as in the second case. Then by $(A_{m,n})$ we have $ys \lessdot \theta(m,n)$. By \Cref{lessdot} and direct inspection we see that the only possibility is $y=\theta(m+1,n-1)$.

Assume $y$ is as in the third case. By $(A_{m,n})$, the term $v^2$ can only occur if \newline $y=\theta(m-1,n-1)$ or $\ell(\theta(m,n))-\ell(y)=2$. But $\theta(m-1,n-1)s>\theta(m-1,n-1)$, so this case does not count.

 Consider the case $\ell(\theta(m,n))-\ell(y)=2$. We can obtain a reduced expression $\underline{y}$ of $y$ by removing two simple reflections in the reduced expression $123\ldots stu=\undtheta(m,n)$. If we do not remove any of the last two simple reflections, then no new braid triple is added when multiplying $\underline{y}$
 by $st$ on the right. This means that $\underline{y}st$ is reduced and $\ell(yst)=\ell(y)+2.$ So at least one of the last two simple reflections must be removed from $\undtheta(m,n)$.
 
 Assume we remove the last $t$ but not the last $u$. Then $yu<y$. If $ys<y$ there exists a reduced expression of $ys$ with the two last letters being $su$, and so by \Cref{stu} we have $yst>ys$, so this case does not occur. 
 
 Hence, we must remove the last $u$. By \Cref{lessdot}, there are at most six different $y\lessdot \theta(m,n)u$ and by inspection one can check that the only one satisfying $y>ys>yst$ is $y=\theta(m-1,n)$.
\end{proof}

We conclude this section by computing which non-perverse summands occur when one multiplies $\undH_{x_n}$ with three simple reflections. This result will be needed in \Cref{illonthewall}.

\begin{definition}
	We say that two elements $h_1,h_2\in \calH$ are \emph{equal up to perverse elements} if there exist $p_1,p_2\in \calH$ perverse such that $h_1+p_1=h_2+p_2$.
\end{definition}

\begin{lemma}\label{notperverse} 
The following equality holds up to perverse elements.

\[ \undH_{x_n} \undH_{s_{n+2}}\undH_{s_{n+1}}\undH_{s_{n+2}}=\begin{cases}
(v+v^{-1}) \undH_{y_{n+1}} + (v+v^{-1}) \undH_{z_{n+1}} & \text{ if $n$ odd}\\
(v+v^{-1}) \undH_{y_{n+1}} & \text{ if $n$ even}
\end{cases}\]
\end{lemma}
\begin{proof}
Recall that we have \[\undH_{s_{n+2}}\undH_{s_{n+1}}\undH_{s_{n+2}}=\undH_{s_{n+2}s_{n+1}s_{n+2}}+\undH_{s_{n+2}}.\] Since $x_ns_{n+2}>x_n$, the term $\undH_{x_n}\undH_{s_{n+2}}$ is perverse.

Similarly, we have $\undH_{s_{n+1}}\undH_{s_{n+2}}\undH_{s_{n+1}}=\undH_{s_{n+2}s_{n+1}s_{n+2}}+\undH_{s_{n+1}}$ and $\undH_{x_n}\undH_{s_{n+1}}$ is also perverse. Hence, up to perverse elements, we have
\[\undH_{x_n} \undH_{s_{n+2}}\undH_{s_{n+1}}\undH_{s_{n+2}}=\undH_{x_n} \undH_{s_{n+1}}\undH_{s_{n+2}}\undH_{s_{n+1}}.\]

We can assume $n$ even, the case $n$ being odd is similar.
Now we can apply \Cref{Heckeonthewall} twice and get
\[\undH_{x_n} \undH_{s_{n+1}}\undH_{s_{n+2}}=\undH_{x_{n+2}}+\undH_{y_{n+1}}+\undH_{z_{n}}\undH_{s_{n+2}}+\undH_{y_n}\undH_{s_{n+2}}.\]
Moreover, by \Cref{out} (more precisely, statements $B_{0,\frac{n-6}{2}}$ and $B_{\frac{n-4}{2},0}$), we have $\undH_{z_{n}}\undH_{s_{n+2}}=\undH_{z_n s_{n+2}}$ and $\undH_{y_n}\undH_{s_{n+2}}=\undH_{y_ns_{n+2}}$. 

The result follows since $y_{n+1}s_{n+1}<y_{n+1}$ while $x_{n+2}s_{n+1}>x_{n+2}$, $z_n s_{n+2}s_{n+1}>z_n s_{n+2}$ and $y_ns_{n+2}s_{n+1}>y_n s_{n+2}$.
\end{proof}

\section{Projectors (\texorpdfstring{\Cref{Thm2}}{Theorem 2})}\label{S3}

\subsection{Diagrammatic Hecke category}

Let $\calH$ be the diagrammatic Hecke category for the Weyl group of type $\tilde{A}_2$ as defined in \cite{EW2} with respect to the Cartan matrix realization over the rational numbers. To be more precise, we consider the three-dimensional realization $\mathfrak{h}=\bigoplus_{s\in S}\alpha_s^{\vee}$ of $(W,S)$ over $\mathbb{Q}$, 
with $\{\alpha_s\, \vert \, s\in S\}\subset \mathfrak{h}^{*}=\Hom_{\mathbb{Q}}(\mathfrak{h}, \mathbb{Q})$ defined by $\langle \alpha_s^{\vee}, \alpha_s\rangle=2$ and $$\langle \alpha_s^{\vee}, \alpha_r\rangle=-1 \ \ \mathrm{if}\ s\neq r.$$
Recall that $R:=S(\mathfrak{h}^{*})$ is the symmetric algebra on $\mathfrak{h}^*$, which we view as a graded $\mathbb{Q}$-algebra with $\mathrm{deg}(\mathfrak{h}^*)=2.$ Denote by $(1)$ the grading shift. For $M,N\in \calH$ let $\Hom^i(M,N)$ denote the degree $i$ morphisms from $M$ to $N$. In other words, $\Hom^i(M,N):=\Hom^0(M,N(i)).$ We call $\Hom$ the set of morphisms of all degrees, i.e., $$\Hom(M,N):=\bigoplus_{i\in \mathbb{Z}}\Hom^i(M,N).$$
For $x\in W$ let $B_x$ be the indecomposable object in $\calH$ corresponding to $x$ and $\cha:\calH \raw H$ be the character map.
 Recall that Soergel's conjecture \cite{EW1} holds for this realization, so we have $\cha(B_x)=\undH_x$.

Let $B,B'\in \calH$.
For $x\in W$ we denote by $\Hom_{<x}(B,B')\subset \Hom(B,B')$ the vector space generated by all morphisms $f:B\raw B'$ that factor through $B_y(n)$ for some $y<x$ and $n\in \bbZ$. Let $\Hom_{\not <x}(B,B'):=\Hom(B,B')/\Hom_{<x}(B,B')$. We denote by $\calH_{\not <x}$ the category whose objects are as in $\calH$ and for any $B,B'\in \calH_{\not <x} $ we have $\Hom_{\calH_{\not <x}}(B,B'):= \Hom_{\not <x}(B,B')$. 

By the Soergel's Hom formula (\cite[\S 6.7]{EW2}), the space $\Hom_{\not <x}(B_x,B)$ is a free graded $R$-module with graded rank given by the coefficient of $\bfH_x$ in $\cha(B)$.
Notice that for any reduced expression $\undx$ of $x$ we have a canonical isomorphism \[\Hom_{\not <x}(B_x,B)\cong \Hom_{\not <x}(BS(\undx),B).\]


Let $x,y\in W$ with $x<y$
and $\undx,\undy$ be reduced expressions for $x,y$.
There is a unique light leaf morphism \cite[Lemma 5.1]{Li3} from $BS(\undy)$ to $BS(\undx)$ of degree $\ell(y)-\ell(x)$, denoted by $G_y^x$. The following is a categorical version of the monotonicity conjecture.

\begin{prop}[{\cite[Prop 5.7]{Pla}}]\label{mono}
	Assume $x<y<z$. Then post-composing with $G_x^y$ induces an injective morphism of left $R$-modules
	\[\Hom_{\not <y}(BS(\undz),BS(\undy))\otimes_R \bbQ \raw \Hom_{\not <x}(BS(\undz),BS(\undx))\otimes_R \bbQ.\]
\end{prop}

\subsection{Induction of light leaves bases}\label{induction}

For general Hecke categories (i.e., for any realization satisfying Soergel's categorification theorem) the indecomposable object $B_x$ is a representative of an equivalence class of isomorphic objects, where
each isomorphism is not canonical. In our setting (where Soergel's
conjecture is available) we have $\End^0(B_x)=\bbQ$, so $B_x$ is a representative of an equivalence
class of isomorphic objects where each isomorphism is canonical up
to an invertible scalar. One can even fix this scalar \cite[Section 3.1]{LW2} so that all the isomorphisms are canonical.

As $\mathcal{H}$ is the Karoubian envelope of the Bott-Samelson Hecke category $\mathcal{H}_{\mathrm{BS}}$, any indecomposable object of $\mathcal{H}$ is (up to shifts) a pair $(BS(\undw), e)$ for $\undw$ an expression and $e$ an idempotent in $\mathrm{End}(BS(\undw))$. Any morphism $f:(B,e)\rightarrow (B',e')$ is, by definition, a map $f: B \rightarrow B'$ satisfying $f=e'\circ f\circ e$.

By the previous discussion, when we write $B_x$ we will mean some pair $(BS(\undx), e_{\undx})$ for $\undx$ a reduced expression of $x$ and $e_{\undx}$ an idempotent in $\mathrm{End}(BS(\undx))$ expressing $B_x$ (we will chose for $e_{\undx}$ the favorite projector defined in \cite[Section 4.1]{Li3}). If the reduced expression $\undx$ is not clear from the context, one may choose any. So a map $B_x\rightarrow B_y$ will be a map $f: BS(\undx)\rightarrow BS(\undy)$ for some reduced expressions $\undx$, $\undy$ of $x$ and $y$ satisfying $f=e_{\undy}\circ f\circ e_{\undx}.$ 

Let $y\leq x$ and $ x\leq xs$. 

\begin{itemize}
\item If $y<ys$, we define $(-)^{U0}:\Hom_{\not<y}(B_x,B_y)\raw \Hom_{\not<y}(B_xB_s,B_y)$
by 
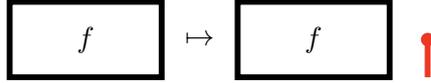
\begin{center}
\tikzset{every picture/.style={line width=2.2pt}}
\begin{tikzpicture}
\draw
(0,0) rectangle ++(2,1);
\node at (1,0.5) {$f$};
\node at (2.5,0.5) {$\mapsto$};
\draw (3,0) rectangle ++(2,1);
\draw[Mred] (5.5,0) edge ++(0,0.5);
\node at (4,0.5) {$f$};
\node[Mred,circle,fill,draw,inner sep=0mm,minimum size=1mm] at (5.5,0.5) {};
\end{tikzpicture}
\end{center}
and $(-)^{U1}:\Hom_{\not<y}(B_x,B_y)\raw \Hom_{\not<ys}(B_{x}B_s,B_{ys})$
by 
\begin{center}
	\begin{tikzpicture}
	\draw
	(0,0) rectangle ++(2,1);
	\node at (1,0.5) {$f$};
	\node at (2.5,0.5) {$\mapsto$};
	\draw (3,0) rectangle ++(2,1);
	\draw[Mred] (5.5,0) edge ++(0,1);
	\node at (4,0.5) {$f$};
	\end{tikzpicture}
\end{center}
Notice that $f^{U1}$ is well-defined since $B_yB_s$ is canonically isomorphic to $ B_{ys}$ in $\calH_{\not <ys}$.

\item If $ys<y$, then we define $(-)^{D0}:\Hom_{\not<y}(B_x,B_y)\raw \Hom_{\not<y}(B_xB_s,B_y)$
by 
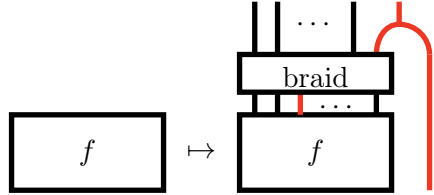
\begin{center}
	\begin{tikzpicture}
	\draw
	(0,0) rectangle ++(2,1);
	\node at (1,0.5) {$f$};
	\node at (2.5,0.5) {$\mapsto$};
	\draw (3,0) rectangle ++(2,1);
	
	\node at (4,0.5) {$f$};
	\draw (3.2,1) edge ++(0,0.3)
		(3.5,1) edge ++(0,0.3)
		(4.8,1) edge ++(0,0.3);
	\draw[Mred] 	(3.8,1) edge ++(0,0.3)
		(5.5,0) edge ++(0,1.8) 
		(5.5,1.8) edge [out=90, in=0] (5.1,2.2)
		(5.1,2.2) edge [out=180, in=90] (4.8,1.8) edge ++(0,0.3); 
	\draw (3,1.3) rectangle ++(2,0.5);
	\node at (4,1.5) {braid};
	\node at (4.3,1.1) {$\cdots$};
	\node at (4,2.2) {$\cdots$};
	\draw (3.2,1.8) edge ++(0,0.7)
	(3.5,1.8) edge ++(0,0.7)
	(4.5,1.8) edge ++(0,0.7);
	\end{tikzpicture}
\end{center}
\end{itemize}
The rectangle labeled ``braid'' is just a sequence of six-valent vertices corresponding to a sequence of braid moves taking the red strand to the right (this sequence always exists). Any braid morphism $BS(\undy)\rightarrow BS(\undy')$ induces the canonical isomorphism $BS(\undy)\cong BS(\undy')$ in $\calH_{\not<y}$. Observe that, since $ys<y$, tensoring on the right with $\Iden_{B_s}$ we also obtain a canonical isomorphism $BS(\undy)B_s\cong BS(\undy')B_s$ in $\calH_{\not<y}$, hence the morphism $f^{D0}$ does not depend on the braid move chosen.

 Similarly, we define
$(-)^{D1}:\Hom_{\not<y}(B_x,B_y)\raw \Hom_{\not<ys}(B_{x}B_s,B_{ys})$
by 
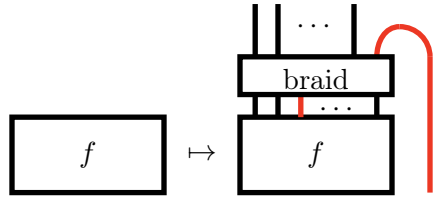
\begin{center}
	\begin{tikzpicture}
\draw
(0,0) rectangle ++(2,1);
\node at (1,0.5) {$f$};
\node at (2.5,0.5) {$\mapsto$};
\draw (3,0) rectangle ++(2,1);

\node at (4,0.5) {$f$};
\draw (3.2,1) edge ++(0,0.3)
(3.5,1) edge ++(0,0.3)
(4.8,1) edge ++(0,0.3);
\draw[Mred] 	(3.8,1) edge ++(0,0.3)
(5.5,0) edge ++(0,1.8) 
(5.5,1.8) edge [out=90, in=0] (5.1,2.2)
(5.1,2.2) edge [out=180, in=90] (4.8,1.8); 
\draw (3,1.3) rectangle ++(2,0.5);
\node at (4,1.5) {braid};
\node at (4.3,1.1) {$\cdots$};
\node at (4,2.2) {$\cdots$};
\draw (3.2,1.8) edge ++(0,0.7)
(3.5,1.8) edge ++(0,0.7)
(4.5,1.8) edge ++(0,0.7);
\end{tikzpicture}
\end{center}

The following lemma was first observed (in a slightly different form) by the first author and G. Williamson in \cite{LW2}.

\begin{lemma}\label{inducinglightleaves}
Let $x,y\in W$	Assume $y \leq x<xs$. Then
	\begin{enumerate}[i)]
		\item if $y<ys$ then
			\[\Hom_{\not <y} (B_x,B_y)(-1) \oplus \Hom_{\not <ys} (B_{x},B_{ys}) \cong \Hom_{\not <y}(B_xB_s,B_y) \]
			and the isomorphism is given by 
			$(f,g)\mapsto f^{U0}+g^{D1}$.
		\item if $ys <y$ then
		\[\Hom_{\not <y} (B_x,B_y)(1) \oplus \Hom_{\not <ys} (B_{x},B_{ys})\cong \Hom_{\not <y} (B_xB_s,B_y) \]
		and the isomorphism is given by 
		$(f,g)\mapsto f^{D0}+g^{U1}$
	\end{enumerate} 
\end{lemma}
\begin{proof}
	We prove only i), the proof of ii) being analogous. 
	The coefficient of $\bfH_y$ in $\undH_x\undH_s$ is $vh_{y,x}(v)+h_{ys,x}(v)$, hence by the Soergel's hom formula the left and the right hand side have the same graded rank.

	Let $Q$ be the fraction field of $R$. For $x\in W$ let $Q_x$ denote the corresponding standard module (cf. \cite[\S 5.4]{EW2}).
	Let $\{f_1,\ldots ,f_n\}$ and $\{g_1,\ldots, g_m\}$ be $R$-bases respectively of $\Hom_{\not <y} (B_x,B_y)$ and $\Hom_{\not <ys} (B_{x},B_{ys})$.
		After localization, they induce $Q$-bases of $\Hom(Q\otimes_R B_x,Q_y)$ and $\Hom(Q\otimes_R B_x,Q_{ys})$.

	We need to show that 
		$\{f_1^{U0},\ldots, f_n^{U0},g_1^{D1},\ldots, g_m^{D1}\}$ is an $R$-linearly independent subset of 
	$\Hom_{\not <y}(B_xB_s,B_y)$.
	It is enough to show that, after localization, this induces a $Q$-linearly independent set in $\Hom(Q\otimes_R B_xB_s,Q_y)$.	Recall that \[Q\otimes B_xB_s\cong (Q\otimes_R B_x)\otimes_Q(Q_e\oplus Q_s),\]
	so we have \begin{equation}\label{Qys}\Hom(Q\otimes_R B_xB_s,Q_y)\cong \bigoplus_{Q_y \cug Q\otimes_R B_x} \Hom(Q_yQ_{e},Q_y)\oplus \bigoplus_{Q_{ys} \cug Q_s\otimes_R B_x} \Hom(Q_{ys}Q_s,Q_y)\end{equation}

	One can conclude with a triangularity argument similar to the proof of \cite[Proposition 6.6]{EW2}: in fact the morphisms $f_i^{U0}$ restrict to zero on the second term of the RHS of \eqref{Qys} since $Q_{ys}$ is not a summand of $Q\otimes B_y$. 
\end{proof}

%
%
%
%
%

\subsection{The projectors on the wall}

We call \emph{retraction} the relation
\begin{center}


\end{center}

 Recall from the introduction that we use the following notation (not used elsewhere in the literature)

\begin{center}

%

\end{center}

The favorite projector for $x_n$ will be denoted by a box labeled by $n$.

In the proof of the following Lemma we are going to use the knowledge of $\ILL(x_n)$ from \Cref{S4}. (Notice that the results in \Cref{S4} do not depend on \Cref{S3}.) 

\begin{lemma}\label{deg0onthewall}
	The morphisms
	\begin{center}

%

\end{center}
are resp. generators of the unidimensional vector spaces $\Hom^0 (B_{y_{n}},B_{x_{n}}B_{s_{n+1}})$
	 and \linebreak
	$\Hom^0(B_{z_{2n}},B_{x_{2n}}B_{s_{2n+1}})$. 
\end{lemma}
\begin{proof}
	Using \Cref{inducinglightleaves}, we know that $\Hom^0(B_{y_{n}},B_{x_{n}}B_{s_{n+1}})=\Hom^1(B_{y_n},B_{x_n})$ and the isomorphism is given by ``applying a D0''. The space $\Hom^1(B_{y_n},B_{x_n})$ is spanned by the unique light leaf of degree one with endpoint $y_n$ with only $U1$ and $U0$.
	
	The other case is similar. We have an isomorphism \[\Hom^1(B_{z_{2n}},B_{x_{2n}})\cong \Hom^0(B_{z_{2n}},B_{x_{2n}}B_{s_{2n+1}})\] given again by $f\mapsto f^{D0}$. The space $\Hom^1(B_{z_{2n}},B_{x_{2n}})$ is one-dimensional. We will show in \Cref{restatement} that a generator of this vector form is of the required form (i.e., it is the element in $\ILL(x_{2n},z_{2n})$ of type III).
\end{proof}

In the diagrammatic category, \Cref{Heckeonthewall}$(B_n)$ and \Cref{deg0onthewall} mean that we can write the first two formulas of \Cref{Thm2} for some numbers $c_n$ and $d_n$.




Our next goal is to compute the coefficients $c_{n}$ and $d_{n}$, that is we need to compute

\begin{equation}\label{intformfigure}

\end{equation}

\subsection{Computing \texorpdfstring{$c_n$}{cn}} 

In the recursive formula for $x_m$ (the projector on the wall) in \Cref{Thm2} we will use the following terminology. There are two equations (for $m$ odd or even). In any case, the term on the left hand side of the equation will be called the \emph{$x$-term}, the term involving an $y_{m-1}$ will be called the \emph{$y$-term} and the term involving $z_{m-1}$ will be called the \emph{$z$-term}.

We start from \eqref{intformfigure} and use the recursive formula for $x_n$ (depending on its parity). Since we are interested in the computation only up to terms smaller than $y_n$, by length reasons (since $\ell(x_{n})-\ell(z_{n-1})=4$) the $z$-term (if it appears) does not add.
Moreover, the $y$-term factors through $B_{y_{n-1}} B_{s_{n+1}}\cong B_{y_{n-1}}(1)\oplus B_{y_{n-1}}(-1)$. So this term also vanishes modulo lower terms. Thus, modulo lower terms only the $x$-term adds. We obtain:

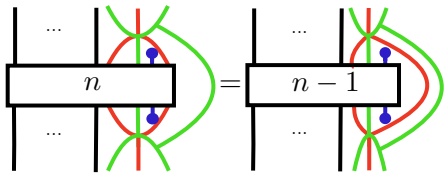
\vspace{-3mm}
\begin{center}
\tikzset{every picture/.style={line width=1.5pt}}

\end{center}
\vspace{-3mm}

Now we apply again the recursive formula for $x_{n-1}$. Again, the $z$-term (if it appears) vanishes modulo lower terms. We have two terms to consider. 

\begin{itemize}
\item The $x$-term is: 
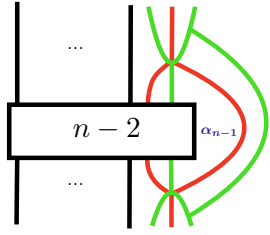
\vspace{-1cm}
\begin{center}

%
\end{center}

\vspace{-3mm}

We can move the polynomial $\alpha_{n-1}$ to the right using twice the nil-Hecke relation (\cite[(5.2)]{EW2}). Notice that breaking the red line gives rise to an expression of length $n-1$ (recall that $\ell(y_n)=n-1$) that is not reduced, so it vanishes modulo lower terms. When we break the green line we get a $-2$ coefficient because $\partial_{n+1}(s_{n}(\alpha_{n-1}))=-2$. After retracting the two green strands, we get a map which is the identity up to lower terms. To see this one needs to write the projector for $x_{n-2}$ in the double leaves basis and see that by length reasons the only terms that survives is the identity. Vertical concatenation of six-valent vertices are the identity modulo lower terms. Summing up, this term is $-2\cdot \mathrm{Id}_{y_n}$ modulo lower terms.

 \item We now consider the $y$-term. It factors through $y_{n-2} s_n s_{n+1}= y_n$.
So this map is $-c_{n-1}$ times the composition of two endomorphisms of $B_{y_n}$, one being the flip of the other. So it suffices to compute only half of it, which is the following morphism:

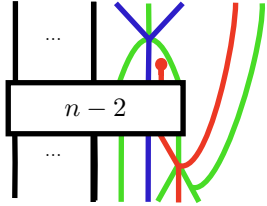
\vspace{-0.7cm}
\begin{center}


\end{center}
\vspace{-0.5cm}
Notice that we have a word of length $n$ in the middle, thus if we apply the recursion for $x_{n-2}$, the $y$ and $z$-terms never appear. So, by induction, we can replace the idempotent for $x_{n-2 }$ with the identity. At the end we obtain (modulo lower terms):

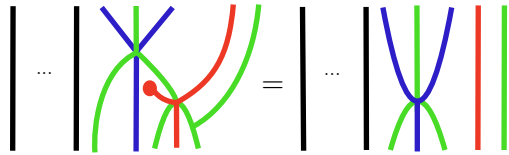
\vspace{-0.2cm}
\begin{center}

%
\end{center}
\vspace{-0.3cm}
After composing with its flip we obtain the identity up to lower terms. So, adding up, this term adds $-c_{n-1}\cdot \mathrm{Id}_{y_n}$ modulo lower terms. 
\end{itemize}

 Therefore we obtain the following recursion:
$c_1=c_2=c_3=c_4=0$ and for $n\geq 4$ we have
\[ \frac{1}{c_{n+1}}=-2 - c_{n-1}.\]
This gives for $n\geq 2$
\[c_{2n+1}=c_{2n+2}=-\frac{n-1}{n}.\]

\subsection{Computing \texorpdfstring{$d_n$}{dn}} 

For later use, we need the following lemma.

\begin{lemma}\label{lambda}
	We have
	$\lambda_{2n}=-\Iden$ modulo lower terms,
	where $\lambda_{2n}\in \mathrm{End}({BS}(\underline{z}_{2n+2}))$ is the following morphism:
	
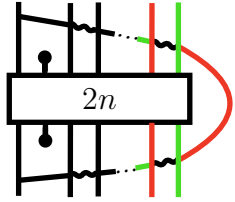
\vspace{-0.7cm}
\begin{center} 


\end{center} 
	
\end{lemma}
\vspace{-1cm}
\begin{proof}
	We apply the recursion for $x_{2n}$ and we claim that only the $x$-term survives.
	In fact, the $z$-term is zero by length reasons. The $y$-term factors through $y_{2n-1} s_{2n-1}$ which is not reduced, so it also vanishes since $\ell(y_{2n-1})<\ell(z_{2n+2})$. We can then apply repeatedly the recursion formula and notice that all the $y$-terms factor through a word of the form 
	\[y^{(m)}:=1234\ldots (m-1)\widehat{m} (m+1) \widehat{(m+2)}(m+3) \ldots 2n (2n-1)\] 
	for $3\leq m\leq 2n-2$. Notice that either  $y^{(m)}$ is not reduced or $s_3$ is not in the left descent set of  $y^{(m)}$ (cf. \Cref{stu}). 
	But $s_3z_{2n+2}<z_{2n+2}$ and so $y^{(m)}\neq z_{2n+2}$.
	This means that by length reasons the $y$-terms, as well as the $z$-terms, always vanish.
	
	Hence, we can replace the idempotent for $x_{2n}$ with the identity. We can now apply the nil-Hecke relations on the first strand, and obtain $\lambda_{2n} =-\Iden$ mod l.t.
\end{proof}

We will compute the intersection form 

\vspace{-.5cm}
\begin{center}


\end{center}
\vspace{-.5cm}
We repeat the strategy for calculating $c_n$. We use the recursion for $x_{2n}$. The $y$-term vanishes. In fact, $B_{y_{2n-1}}B_{s_{2n+1}}$ is indecomposable by the last equation of \Cref{Thm1}.ii) (recall that $y_{2n-1}\sim s_3\theta(n-3,0)$) and  $z_{2n}\neq y_{2n-1}s_{2n-1}$ (they have different lengths). There is no $z$-term in this formula, so our intersection form is the same as 

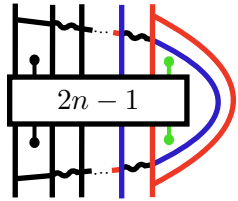
\vspace{-0.5cm}
\begin{center} 
%

\end{center}
\vspace{-0.5cm}

We apply the recursion and obtain three terms. 
\begin{itemize}
\item \textbf{$x$-term:} We need to calculate
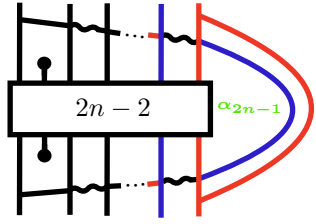
\vspace{-0.8cm}
\begin{center}

%

\end{center} 
\vspace{-1cm}
By using the nil-Hecke relation twice, we obtain three terms: ``cutting the blue line'', ``cutting the red line'' and ``the polynomial outside''. 
\begin{itemize}
\item Cutting the blue line gives you 

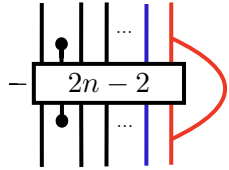
\vspace{-0.5cm}
\begin{center} 
%
\end{center} 
The reason for this is that, when we apply the relation

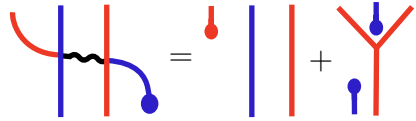
\begin{center}
%
\end{center} 

 to this situation, the second term does not appear because it factors through something too small. So one can retract the blue dot and pass through the six-valent vertices and at the end use the retraction relation. The map that we obtain is zero modulo lower terms. 
 Otherwise, $B_{z_{2n}}$ would be a summand of $B_{x_{2n-2}}B_{s_{2n-2}}\cong 
 B_{x_{2n-2}}(1)\oplus B_{x_{2n-2}}(-1)$.
 
\item Cutting the red line gives us a coefficient $-2 \lambda_{2n-2}$, which is equal to $2$ by \Cref{lambda}. This will give us the $2$ in \eqref{dn}.
\item The last one is zero by degree reasons (the picture would give a negative degree map, which is always zero modulo lower terms).
\end{itemize}

\item \textbf{$y$-term:} This term is zero because $B_{y_{2n-2}}B_{s_{2n}}B_{s_{2n+1}}\cong B_{\theta(n-4,1)}\oplus B_{\theta(n-4,0)} \oplus B_{\theta(n-5,0)}$ from \Cref{out} (recall that $y_{2n-2}=\theta(n-4,0)$). There cannot be morphisms of degree $0$ since $z_{2n}\not\in \{\theta(n-4,1), \theta(n-4,0),\theta(n-5,0)\}$.

\item\textbf{$z$-term} This term is a composition of a morphism and its flip, but in the middle one has $z_{2n-2}s_{2n}s_{2n+1}$, which is the same element as $z_{2n}$. Thus we can replace the idempotent on $z_{2n-2}$ by the identity in this calculation. Thus, this term adds $-d_{2n-1}\beta^2$ , where $\beta$ is the coefficient of the identity of the following morphism modulo lower terms. 

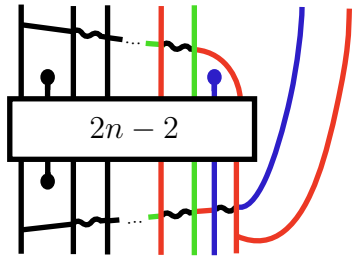
\vspace{-1cm}
\begin{center}


\end{center} 

\vspace{-1cm}	
(At the end of this calculation we will see that $\beta^2=1$ so this term will contribute with $-d_{2n-1}$ in \eqref{dn}.)
We apply the recursion for $x_{2n-2}$. Again by length reasons there is no $z$-term. Let us consider the $y$-term. It factors through $y_{2n-3}s_{2n}s_{2n+1}$, but $s_{2n}$ and $s_{2n+1}$ are in the right descent set of $y_{2n-3}$, so we obtain that 
$$B_{y_{2n-3}}B_{{2n}}B_{{2n+1}}\cong B_{y_{2n-3}}(2)\oplus B_{y_{2n-3}} \oplus B_{y_{2n-3}}\oplus B_{y_{2n-3}}(-2).$$ 
As $\ell(y_{2n-3})<\ell(z_{2n})$ this term vanishes. 

So we only have to care about the $x$-term, that is:
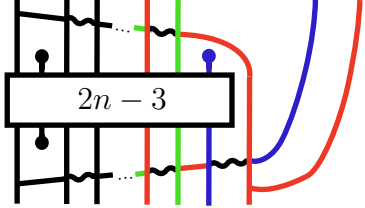
\vspace{-1cm} 
 
\begin{center}



\end{center} 

\vspace{-1cm}
We apply for a last time the recursion. This time we get three terms. The $z$-term vanishes because is too small. The $y$-term also vanishes: it factors through $y_{2n-4} s_{2n-2}s_{2n} s_{2n+1}$, and since $y_{2n-4} s_{2n-2}s_{2n}$ is reduced and $s_{2n+1}$ is in its right descent set, all of its terms are smaller or equal than $y_{2n-4} s_{2n-2}s_{2n}$. Since $$\ell(y_{2n-4} s_{2n-2}s_{2n})=\ell(z_{2n})$$ and $y_{2n-4} s_{2n-2}s_{2n} \neq z_{2n},$ this term vanishes.

The remaining is the $x$-term:
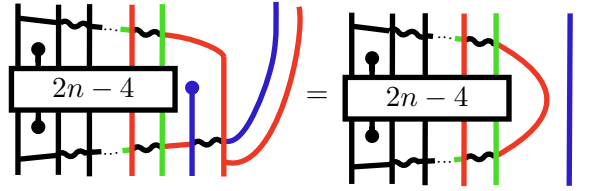
\vspace{-0.7cm}

\begin{center}

%

\end{center}

\vspace{-1cm}
which is $\lambda_{2n-4}$. Thus $\beta=-1$. 
\end{itemize}

We have $d_1=d_3=d_5=0$.
For $n\geq 3$ we obtain the following recursive formula for $d_n$.
\begin{equation}\label{dn}
\frac{1}{d_{2n+1}}=2-d_{2n-1},
\end{equation}
that is solved by
\[ d_{2n+1} =\frac{n-2}{n-1},\]
and with this we finish the proof of \Cref{Thm2} for the elements on the wall.

With basically the same strategy and techniques one can prove \cref{Thm2} beyond the walls. In fact, one finds the following recursive formulas for the coefficients $c_{m}$ and $d_{m,n}$ in \eqref{projoow}.
\[c_1=0\qquad \text{ and }\quad-\frac{1}{c_{m+1}} = -2 + c_m\quad \text{ for } m \geq 1\]
\[-\frac{1}{d_{1,n}}= 3- 2 c_n \qquad  \text{ and } \quad
-\frac{1}{d_{m+1,n}} = 4 - 2c_n + d_{m,n} (2-c_n)^2\quad  \text{ for }m\geq 1\ \]
which are solved by 
\[c_{m+1} =\frac{m}{m+1}\qquad \text{and}\qquad d_{m+1,n}=-\frac{n(n+m+1)}{(n+1)(n+m+2)}.\]

We omit the proof because writing down the whole calculation would make this paper too long.

\section{Indecomposable light leaves (\texorpdfstring{\Cref{Thm3}}{Theorem 3})}\label{S4}

\subsection{ILL on the wall}\label{illonthewall}

The goal of this section is to prove \Cref{Thm3} on the wall. Recall from the introduction that for every $n$ the set $\mathrm{ILL}(x_n)$ is composed by the light leaves
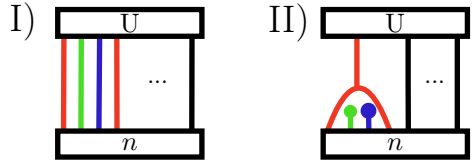
\begin{center}



\end{center}

and if $n$ is even, also by the two maps
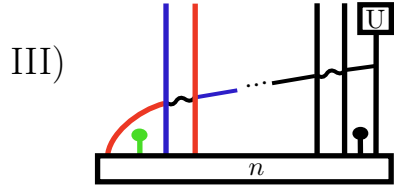
\begin{center}

	%

\end{center} 

\begin{definition}
	For every $y\leq x$ let $\ILL(x,y)$ be the subset of light leaves in $\ILL(x)$ with endpoint a reduced expression for $y$. 
\end{definition}

The goal of this section is to show the following proposition for $x=x_n$.
\begin{prop}\label{restatement} For any $y,x$ with $y\leq x$ the light leaves in $\ILL(x,y)$ form a basis of $\Hom_{\not < y}(B_x,B_y)$.
\end{prop}

\Cref{restatement} implies \Cref{Thm3} by a well-known standard argument (see for example \cite[\S 7.3]{EW2}).

Recall the definition of a $01$-sequence from \cite[\S 2.4]{EW2}.
We now start by classifying all the $01$-sequences for $\underline{x_n}$ with a $D$ (or Down) in the last position, and such that only $U$ (or Up) occur before the last position. 

\begin{lemma}\label{tree}
	Let $\epsilon$ be a $01$-sequence for $\underline{x_n}$. Assume that $\epsilon_i$ is a $U0$ or a $U1$ for any $i<n$ and that $\epsilon_n$ is a $D0$ or a $D1$. Then $\epsilon$'s ending is one of the following (where $*$ means it can be either $0$ or $1$).
	\begin{enumerate}
		\item $100*$
		\item $1101*$ 
		\item $10\overbrace{11\ldots 1}^{2k}0*$.
	\end{enumerate}
\end{lemma}

\begin{proof}

Let $\calA$ be the set of $01$-sequences $\epsilon$ for $\undx_n$ such that $\epsilon_i$ is a $U0$ or a $U1$ for any $i<n$ and that $\epsilon_n$ is a $D0$ or a $D1$.
The following tree shows all the possible endings of a $01$-sequence. We will prove that the sequences which are in colored boxes (by blue, red, green and yellow) cannot be ending of elements of $\calA$.
This means that the only possible endings for the sequences in $\calA$ are the ones in the white boxes, as wished.
	\begin{center}
		
\tikzset{every picture/.style={line width=0.7pt}}		
	\begin{tikzpicture}[scale=0.75]
	\node (d) at (0,0) {$*$};
	\node (1d) at (5,-1) {$1*$};
	\node (0d) at (-5,-1) {$0*$};
	\node (01d) at (3,-2) {$01*$};
	\node (11d)at (7,-2)[rectangle, draw, fill=Mgreen!20] {$11*$};
	\node (001d) at (2,-3)[rectangle, draw, fill=Mblue!20] {$001*$};
	\node (101d) at (5,-3) {$101*$};
	\node (1101d) at (3,-4)[rectangle, draw] {$1101*$};
	\node (0101d) at (7,-4) {$0101*$};
	\node (00101d) at (5.5,-5)[rectangle, draw, fill=Mblue!20] {$00101*$};[rectangle, draw, fill=blue!20]
	\node (10101d) at (8.5,-5) [rectangle, draw, fill=yellow!20]{$10101*$};
	
	\node (00d) at (-8,-2) {$00*$};
	\node (10d) at (-2,-2) {$10*$};
	\node (000d) at (-9,-3)[rectangle, draw, fill=Mblue!20] {$000*$};
	\node (100d) at (-7,-3) [rectangle,draw] {$100*$};
	\node (010d) at (-5,-3) {$010*$};
	\node (110d) at (-1,-3) {$110*$};
	\node (0010d) at (-6,-4)[rectangle, draw, fill=Mblue!20] {$0010*$};
	\node (1010d) at (-4,-4)[rectangle, draw, fill=Mgreen!20] {$1010*$};
	\node (2k) at (-3.5,-5) {$0\overbrace{11\ldots 1}^{2k}0*$};
	\node (2k+1) at (1,-5) {$0\overbrace{11\ldots 1}^{2k+1}0*$};
	\node (002k) at (-6.5,-7)[rectangle, draw, fill=Mblue!20] {$00\overbrace{11\ldots 1}^{2k}0*$};
	\node (102k) at (-3,-7)[rectangle, draw] {$10\overbrace{11\ldots 1}^{2k}0*$};
	\node (002k+1) at (0.5,-7)[rectangle, draw, fill=Mblue!20] {$00\overbrace{11\ldots 1}^{2k+1}0*$};
	\node (102k+1) at (4,-7)[rectangle, draw, fill=Mred!20] {$10\overbrace{11\ldots 1}^{2k+1}0*$};
	\path[-]
	(d) edge (1d) edge (0d)
	(1d) edge (01d) edge (11d)
	(01d) edge (001d) edge (101d)
	(101d) edge (1101d) edge (0101d)
	(0101d) edge (00101d) edge (10101d)
	
	(0d) edge (00d) edge (10d)
	(00d) edge (000d) edge (100d)
	(10d) edge (010d) edge (110d)
	(010d) edge (0010d) edge (1010d)
	(110d) edge (2k) edge (2k+1)
	(2k) edge (002k) edge (102k)
	(2k+1) edge (002k+1) edge (102k+1)
	;
	\end{tikzpicture}
	\end{center}

Each different color of the boxes represents a different reason why that sequence cannot be the ending of an element in $\calA$.
\begin{itemize}
\item The sequence in the red box because \Cref{redex} implies that $\epsilon_n$ is $U$.
 \item Sequences in a green box because \Cref{stu} implies that $\epsilon_n$ is $U$.
 \item Consider a sequence in a blue box. Let $w$ be the element expressed by the 
 sequence \emph{before} using the part of the sequence in the box.
 All the sequences in blue boxes begin with $001$ or $000$. This means that there is no simple reflection in the right descent set of $w$, so we must have $w=e$. But if this happens, then $\epsilon_n$ must be a U (this last part of the argument is an easy box by box  check).
\item Consider the sequence in the yellow box. Label the simple reflections indexed in the box by $stustu$. Let $w$ be defined as in the previous point. We have $wst>ws$ and $wsu>ws$, hence $ws$ is minimal in its right $\{t,u\}$-coset and $wstut>wsut$, forcing $\epsilon_n$ to be a $U$. \qedhere
\end{itemize}
\end{proof}

The next Lemma is the crucial observation of this section. 
\begin{lemma}\label{1101d}
For every $n\geq 5$ we have
	\begin{center}


	\end{center}
\end{lemma}
\begin{proof}

Recall that the favorite projector $e_{\undx_n}\in \mathrm{End}(BS(\undx_n))$ defining $B_{x_n}$ satisfies the \emph{absorption property}, i.e., for every $m\leq n$ we have $e_{\undx_n}=(e_{\undx_m}\otimes \Iden)\circ e_{\undx_n}$, or in diagrams:
\begin{center}
%
\end{center}

	Thus $\phi_n$ factors through a degree $0$ morphism from $B_{x_n}$ to $B_{x_{n-5}}B_{s_n}B_{s_{n-1}}B_{s_n}$. 
	From \Cref{notperverse} we know that \[B_{x_{n-5}}B_{s_n}B_{s_{n-1}}B_{s_n}\cong \begin{cases} B'\oplus B_{y_{n-4}}(1)\oplus B_{y_{n-4}}(-1)& \text{ if }n\text{ even}\\
	B'\oplus B_{y_{n-4}}(1)\oplus B_{z_{n-4}}(1)\oplus B_{y_{n-4}}(-1)\oplus B_{z_{n-4}}(-1) &\text{ if }n\text{ odd }\end{cases}\]
	where $B'$ is a perverse object.
	
	Since $B_{x_n}$ is not a summand of $B_{x_{n-5}}B_{s_n}B_{s_{n-1}}B_{s_n}$, the image of $\phi_n$ must be contained in $B_{y_{n-4}}(1)$ if $n$ even and in $B_{y_{n-4}}(1)\oplus B_{z_{n-4}}(1)$ if $n$ odd. By Soergel's hom formula, if $y<x_n,$ there exists a non-trivial map of degree one between $B_{x_n}$ and $B_y$ if and only if $\mu(y,x_n)\neq 0$, and this can happen only if $\ell(x_n)-\ell(y)\leq 3$ by part A of \Cref{Heckeonthewall}.
	Since $\ell(z_{n-4})<\ell(y_{n-4})=\ell(x_n)-5$ we conclude that $\phi_n=0$.
\end{proof}

\begin{cor}\label{movingcaps}
	For every $n\geq k+4$ we have
\vspace{-5mm}

\begin{center}

\end{center}

\end{cor}
\begin{proof}
	By passing the dot through the six-valent vertex, we obtain:
	
	\vspace{-8mm}
	\begin{center}
\begin{equation}\label{onemorphismtwocaps}
%
\end{equation}

\end{center}

	After composing with the projector, the LHS in \eqref{onemorphismtwocaps} vanishes by \Cref{1101d}. Now we can easily conclude by induction on $k$.
\end{proof}
	
\begin{lemma}\label{match}
For all $n\in \mathbb{N}$ and $y\in W$ with $y\leq x_n$, the graded ranks of $\ILL(x_{n},y)$ and of $\Hom_{\not <y}(B_{x_{n}},B_{y})$ match up.
\end{lemma}
\begin{proof}
By Soergel's hom formula, if $(-,-):H\times H\rightarrow \mathbb{Z}[v,v^{-1}]$ is the $\mathbb{Z}[v,v^{-1}]-$bilinear pairing defined by $(\bfH_x,\bfH_y)=\delta_{x,y}$ (the Kronecker delta), then the graded rank of \newline
 $\Hom_{\not <y}(B_{x_{n}},B_{y})$ is given by $(\underline{\bfH}_{x_n}, \bfH_y).$ We can conclude by \Cref{Heckeonthewall}.
\end{proof}

\begin{lemma}\label{vanishing}
Let $y\leq x_n$. If $f$ is a light leaf morphism $f: BS(\undx_n)\raw B_y$ with two or more D's, then $f\circ e_{\undx_n}=0$ in $\Hom_{\not <y}(B_{x_{n}},B_{y})$ 
\end{lemma}
\begin{proof}
As the graded rank of 
 $\Hom_{\not <y}(B_{x_{n}},B_{y})$ is given by $(\underline{\bfH}_{x_n}, \bfH_y),$ by \Cref{Heckeonthewall} we have that $\Hom_{\not <y}(B_{x_{n}},B_{y})$ has a basis consisting of morphisms of degree greater or equal than $\ell(x_n)-\ell(y)-2$. If a light leaf has two or more D's, then 
 its degree is lesser or equal than $\ell(x_n)-\ell(y)-4$.
\end{proof}

We are now ready to prove the main result of this section.

\vspace{0.2cm}

\emph{Proof of \Cref{restatement}}
We assume the claim for $x_n$ and show it for $x_{n+1}$.
	We have $\undx_{n+1}=\undx_n s_{n+1}$. 
	Let $y\leq x_n$ and assume that $ys_{n+1}>y$. 
	We want to show that $\ILL(x_{n+1},y)$ and $\ILL(x_{n+1},ys_{n+1})$ are bases respectively of $\Hom_{\not <y}(B_{x_{n+1}},B_y)$
	and of $\Hom_{\not <ys_{n+1}}(B_{x_{n+1}},B_{ys_{n+1}})$. By \Cref{match}, the graded ranks match up, so it is enough to show that they generate as $R$-modules. To do so, for the rest of this proof we will use the following strategy.
	
	\vspace{0.5cm}
\textbf{\emph{Strategy:}}	 \emph{We will show that after precomposing with the idempotent $e_{\undx_{n+1}}$ every map in
	$\Hom_{\not <y}(B_{x_n}B_{s_{n+1}},B_y)$
	and in $\Hom_{\not <ys_{n+1}}(B_{x_n}B_{s_{n+1}},B_{ys_{n+1}})$ lies in the $R$-span of the set $\ILL(x_{n+1})$.}
	
	\vspace{0.5cm}
	
	Let $\ILL(x_n,y)=\{f_1,\ldots,f_p\}$ and $\ILL(x_n,ys_{n+1})=\{g_1,\ldots,g_q\}$. Then, by \Cref{inducinglightleaves} the sets $\{f_1^{U0},\ldots,f_p^{U0},g_1^{D1},\ldots,g_q^{D1}\}$ and $\{f_1^{U1},\ldots,f_p^{U1},g_1^{D0},\ldots,g_q^{D0}\}$ are bases respectively of 
	 $\Hom_{\not <y}(B_{x_n}B_{s_{n+1}},B_y)$
	and of $\Hom_{\not <ys_{n+1}}(B_{x_n}B_{s_{n+1}},B_{ys_{n+1}})$.

	\begin{itemize}	
\item	Consider the case $b\in \ILL(x_n,y)$ (i.e., $b=f_i$ for some $i \in \{1,2,\ldots, p\})$.
	If $b$ is of type I or II then by the absorption property, when precomposing $b^{U0}$ and $b^{U1}$ with the idempotent $e_{\undx_{n+1}}$ we obtain again an element of type I or II of $\ILL(x_{n+1})$. On the other hand, $b$ cannot be of type III since we required that $ys_{n+1}>y$, thus we fulfilled our strategy for $b\in \ILL(x_n,y)$.
	
\item	Consider now the case $b\in \ILL(x_n,ys_{n+1})$. Let $b$ be of type II or III. By \Cref{vanishing} the two maps $b^{D0}$ and $b^{D1}$ must be zero. 
It remains to consider only the case
when $b$ is of type I. As a consequence of \Cref{tree}, there are three cases to consider.
	\begin{itemize}
		\item In the case ending with $100*$ the two induced maps $b^{D0}$ and $b^{D1}$ precomposed by $e_{\undx_{n+1}}$ are of type II by \Cref{movingcaps}.
		\item The case $1101*$ gives rise to morphisms $b^{D0}$ and $b^{D1}$ that vanish if one precomposes with $e_{\undx_{n+1}}$ by \Cref{1101d}. 
		\item The remaining case $10\overbrace{11\ldots 1}^{2k}0*$ is divided into two sub-cases. If $n=2k+4$ then $b^{D0}$ and $b^{D1}$ precomposed by $e_{\undx_{n+1}}$ are precisely the two light leaves of type III in $\ILL(x_{n+1})$.
		
		\begin{claim}\label{claim}
				If $n>2k+4$ the maps $b^{D0}$ and $b^{D1}$ precomposed by $e_{\undx_{n+1}}$ are linear combinations of elements of type II in $\ILL(x_{n+1})$.
		\end{claim}
		\begin{proof}[Proof of the claim.]
		We show the claim by induction on $k$. In the position just before $10\overbrace{11\ldots 1}^{2k}0D$ we cannot have a $1$ since $b$ is of type I (and the subexpression $11011$ cannot be part of a U-sequence).
		This means that $$b^{D0}\in \Hom_{\not <{ys_{n+1}}}(B_{x_{n+1}},B_{ys_{n+1}})$$ (an almost identical argument works for $b^{D1}$) is some map tensored on the right with something of the following form (here drawn for $k=2$).
		\begin{center}



\end{center}

		We apply the following relation in $\Hom(B_s,B_sB_s)$ in the upper left corner of the last diagram.

\begin{equation}\label{dotswap}
%
\end{equation}

		We obtain three terms, one for each term in the RHS of \eqref{dotswap}.The last term vanishes modulo lower terms (i.e., elements $<ys_{n+1}$).

		 The term with a trivalent vertex also vanishes modulo lower terms. To see this,
		 		 notice that one can take the polynomial $\alpha$ to the extreme left of the picture by repeatedly using the nil-Hecke relation, and realizing that in that relation the term with two dots, does not contribute because it factors through lower terms. So one obtains that this morphism is equal to a degree two polynomial multiplied by a light leaf with two D's which must be zero by \Cref{vanishing}.
		 
So only the first term (a line and a dot) remains:
		\begin{center}



\end{center}
		We use \cite[Relation 5.7]{EW2} (pulling a dot through a six-valent vertex) and we obtain modulo lower terms:
		\begin{center}

%
		\end{center}
		
		The term on the right is of type II. The term on the left is of the same type as the one we started with, but with $k$ replaced by $k-1$. (If $k=1$ we obtain two terms of type II.) By the induction hypothesis, that term is a linear combination of maps of type II modulo lower terms. This finishes the proof of \Cref{claim} and of \Cref{restatement}.\end{proof}

	\end{itemize}

\end{itemize}

\subsection{ILL Beyond the wall}

Every element $w\in W$ beyond the walls is either $w\sim \theta(m,n)$, $w\sim\theta(m,n)s$,
$w\sim r\theta(m,n)$ or $w\sim r\theta(m,n)s$ (here $s$ is the unique simple reflection not in the right descent set of $\theta(m,n)$ and $r$ the simple reflection not in the left descent set of $\theta(m,n)$). Recall from \Cref{Thm1} that we have \[\undH_{\theta(m,n)s}=\undH_{\theta(m,n)}\undH_s, \quad\undH_{r\theta(m,n)}=\undH_r\undH_{\theta(m,n)}\quad\text{ and
}\quad\undH_{r\theta(m,n)s}=\undH_r\undH_{\theta(m,n)}\undH_s\]
 hence also 
 \[B_{\theta(m,n)s}\cong B_{\theta(m,n)}B_s, \quad B_{r\theta(m,n)}=B_r B_{\theta(m,n)}\quad \text{ and
 }\quad B_{r\theta(m,n)s}=B_rB_{\theta(m,n)}B_s\]

In view of the induction of light leaves explained in \Cref{induction}, it is enough to understand $\ILL(\theta(m,n))$. For example, $\ILL(\theta(m,n)s)$ can be obtained inducing from $\ILL(\theta(m,n))$ since $B_{\theta(m,n)s}\cong B_{\theta(m,n)}B_s$. Similarly, we can obtain $\ILL(r\theta(m,n))$ and $\ILL(r\theta(m,n)s)$ by applying the induction of light leaves (resp. from $\ILL(\theta(m,n))$ and $\ILL(\theta(m,n)s)$) \emph{on the left}.

We begin with two preparatory observations.
The following is a consequence of \Cref{out}.
\begin{lemma}\label{boundsonmorphism}
	Let $y\leq \theta(m,n)$ and $\ell(\theta(m,n))-\ell(y)=4a+b$ with $0\leq b\leq 3$.
	Then $\Hom^k(B_{\theta(m,n)},B_y)=0$ for any $k< 2a+b$.
	
	Moreover, $\dim_\bbQ \Hom^{2a}(B_{\theta(m,n)},B_{\theta(m-a,n-a)})=1$.
\end{lemma}
\begin{proof}
	We will prove the first claim by induction on $\ell(y)$.
	By \Cref{out} and Soergel's hom formula, we can see that $\Hom^k_{\not <y}(B_{\theta(m,n)},B_y)=0$ for any $k< 2a+b$. 
	In fact, if $v^k\bfH_y$ occurs in $v^{2i}\bfN_{\theta(m-i,n-i)}$, then $a\geq i$ and $k=2i+4(a-i)+b\geq 2a +b$.
	This gives, in particular, the base case $y=e$.

	Assume now that $0\neq f\in \Hom^k (B_{\theta(m,n)},B_y)$ factors through $B_z$ for some $z<y$, i.e., we have $f=g\circ h$ where $h\in \Hom^{k_1}(B_{\theta(m,n)},B_z)$, $g\in \Hom^{k_2}(B_z,B_y)$,  and $k_1+k_2=k$. Since $y\neq z$ we have $k_2\geq 1$. By induction we have that if $\ell(\theta(m,n))-\ell(z)=4a'+b'$ then $k_1\geq 2a'+b'$. Since $4a'+b'>4a+b$ implies $2a'+b'+1\geq 2a+b$ we obtain $k\geq k_1+1\geq 2a+b$ as desired.
	
	For the second claim notice that $\dim_\bbQ \Hom^{2a}_{\not <y}(B_{\theta(m,n)},B_{\theta(m-a,n-a)})=1$ and that $\Hom^{k}_{\not <y}(B_{\theta(m,n)},B_y)=0$ if $y<\theta(m-a,n-a)$ for any $k\leq 2a$ by the first part.
\end{proof}

The following simple Lemma will be used to detect many morphisms that vanish modulo lower terms.

\begin{lemma}\label{fac1}
	Let $y,z\in W$ with $y\lessdot z$ and let $\undy$ and $\undz$ be reduced expressions for $y$ and $z$. Consider a morphism
	\[ f : B_{\theta(m,n)}\raw BS(\undy).\]
	Assume that $f$ factors through $BS(\undz)$ as $f=G\circ h$, with $h:B_{\theta(m,n)}\raw BS(\undz)$ and  $G$ of degree one. Then $f$ factors through $B_z$  in $\calH_{\not<y}$, i.e., we have $f=G\circ h'$ with $h':B_{\theta(m,n)}\raw B_z$ and $G$ of degree one.
	
	In particular, if $y=\theta(m-a,n-a)$ and $f$ is of degree $2a$ then $f=0$ in $\mathcal{H}_{\not <y}$.
\end{lemma}
\begin{proof}
	Notice that $BS(\undy)\cong B_y$ in $\calH_{\not<y}$. We have 
	\begin{equation}\label{bsz}
BS(\undz) \cong B_z \oplus \bigoplus_{w<z} B_w^{\oplus h_w(v)}.
	\end{equation}
For every $w$ occurring in the  direct sum in \eqref{bsz} we have $w\not \geq y$, hence every morphism $B_w\raw B_y$ vanishes in $\calH_{\not < y}$. It follows that, if $e_{\undz}$ is an idempotent for $B_z$ in $BS(\undz)$, we have $f = G\circ e_{\undz} \circ h$, and the first claim follows by taking $h'=e_{\undz} \circ h$.  

For the second claim notice that, if $z\gtrdot \theta(m-a,n-a)$ then $\ell(\theta(m,n))-\ell(z)=4a-1$ and $\Hom^{2a-1}(B_{\theta(m,n)},B_z)=0$ by \Cref{boundsonmorphism}.
\end{proof}

For $0\leq i\leq \min (m,n)$, let $\Psi_{m,n}^i:BS(\undtheta(m,n))\raw BS(\undtheta(m-i,n-i))$ denote the following morphism of degree $2i$.

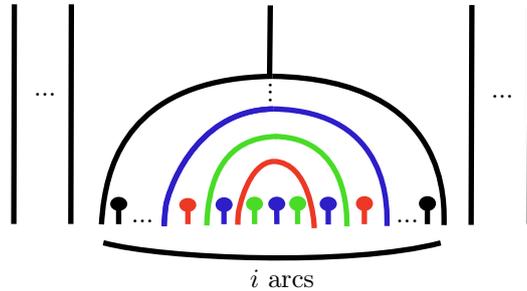
\begin{center}



\end{center} 
Let $\bar{\Psi}^i_{m,n}:=\Psi^i_{m,n}\circ e_{\undtheta(m,n)}:B_{\theta(m,n)}\raw BS(\undtheta(m-i,n-i))$.

\begin{prop}\label{lower}
	Let $m,n \in \mathbb{Z}_{\geq 0}$. Assume that, for any $0\leq i\leq \min (m,n)$, the morphism $\bar{\Psi}^i_{m,n}\neq 0$ up to lower terms. Then $\ILL(\theta(m,n),y)$ is a basis of $\Hom_{\not < y}(B_{\theta(m,n)},B_y)$ for every $y\leq \theta(m,n)$.
\end{prop}
\begin{proof}

	 If $\bar{\Psi}^i_{m,n}\neq 0$ up to lower terms, then by \Cref{mono}, all the light leaves in $\ILL(\theta(m,n))$ are non-zero up to lower terms since they can be obtained via post-composition with $G^{\theta(m,n)}_y$. Moreover, since for every $y\leq \theta(m,n)$ the light leaves in $\ILL(\theta(m,n),y)$ sit in different degrees, they are linearly independent over $R$. Since the graded rank of $\ILL(\theta(m,n),y)$ matches the graded rank of $\Hom_{\not < y}(B_{\theta(m,n)},y)$, it follows that $\ILL(\theta(m,n),y)$ is a basis. 
\end{proof}

For any $m,n \in \mathbb{Z}_{\geq 0}$, let us denote by $(X_{m,n})$ the following statement: for all $0\leq i\leq \min (m,n)$, the morphism $\bar{\Psi}^i_{m,n}\neq 0$ up to lower terms.

To finish the proof of \Cref{Thm3}, we just need to prove $(X_{m,n})$ for all $m,n \in \mathbb{Z}_{\geq 0}$. We will prove it for $n\leq m$ (the case $n\geq m$ being analogous). The proof will be by induction on $(m,n)$. More precisely, for all $m\in \mathbb{Z}_{\geq 0}$, $X_{m,0}$ is trivial. If $n>0$ we will prove $X_{m,n}$ assuming the knowledge of $X_{m',n'}$ for all $m'\leq m$ and $n'<n$.


There are two slightly different cases to consider. 
\begin{enumerate}
	\item $i=n,$
	\item $i<n$.
\end{enumerate}

\subsubsection*{Case \texorpdfstring{$i=n$}{i=n}}

By induction hypothesis and \Cref{lower} we know that $\ILL(\theta(m,n-1),y)$ is a basis of $\Hom_{\not < y}(B_{\theta(m,n-1)},B_y)$ for every $y\leq \theta(m,n-1)$. Recall that $\overline{\Psi}_{m,n}^n$ is an element of $\Hom^{2n} (B_{\theta(m,n)},B_{\theta(m-n,0)})$.

\begin{lemma}\label{threelittleleaves}
We have 
\[\dim_{\mathbb{Q}} \Hom^{2n}_{\not < \theta(m-n,0)} (B_{\theta(m,n-1)}B_sB_t,B_{\theta(m-n,0)})=\begin{cases} 2 &\text{ if }n=1\\
3 & \text{ if }n>1.\end{cases}.\]	
and a basis of 	$ \Hom^{2n}_{\not < \theta(m-n,0)} (B_{\theta(m,n-1)}B_sB_t,B_{\theta(m-n,0)})$ is given by the following three elements (where the number of arcs are $n$, $n-1$ and $n-2$, respectively):

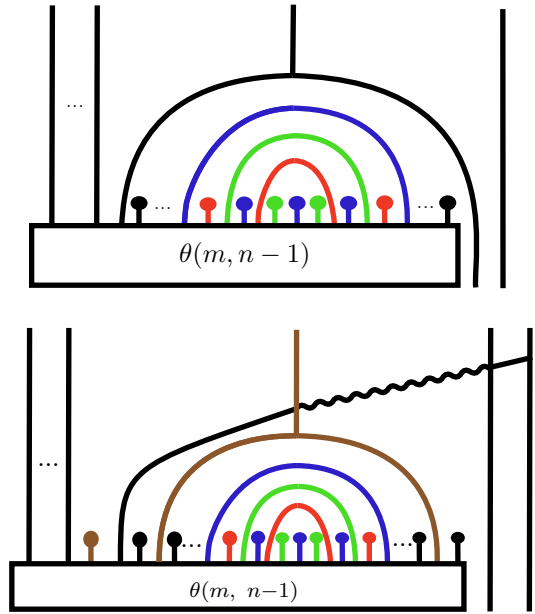
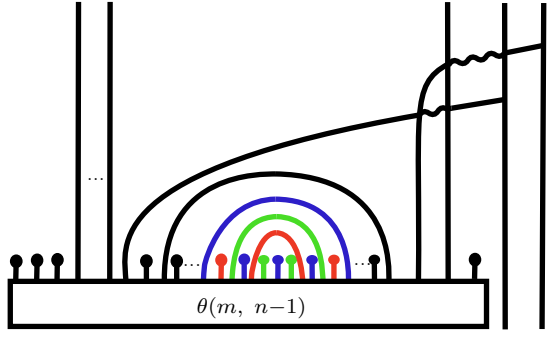
\begin{center}


\end{center}

\end{lemma}
\begin{proof}

It is easy to show using \Cref{out}(3) that the coefficient of $v^{2n}\bfH_{\theta(m-n,0)}$ on $\undH_{\theta(m,n-1)}\undH_s\undH_t$ is 
		 $3$ for $n>1$ (and $2$ for $n=1$ since in this case the term $\undH_{\theta(m-1,0)}$ does not contribute to $v^2\bfH_{\theta(m-1,0)}$). Thus we have proved the dimension formula.

		 The three listed leaves can be induced from $\ILL(\theta(m,n-1))$ using the procedure of \Cref{inducinglightleaves} (the third one does not occur for $n=1$), so they are part of an $R$-basis of $ \Hom_{\not < \theta(m-n,0)} (B_{\theta(m,n-1)}B_sB_t,B_{\theta(m-n,0)})$, thus linearly independent over $\mathbb{Q}$ and all belonging to $ \Hom^{2n}_{\not < \theta(m-n,0)} (B_{\theta(m,n-1)}B_sB_t,B_{\theta(m-n,0)})$.
\end{proof}

By \Cref{boundsonmorphism} we know that 
\begin{equation}\label{dim1}
\dim_{\mathbb{Q}} \Hom^{2n}(B_{\theta(m,n)},B_{\theta(m-n,0)})=1,
\end{equation}
 and this space is generated by the three leaves in \Cref{threelittleleaves} (that we will call $l_1, l_2$ and $l_3$) composed with the favorite projector $e_{\undtheta(m,n)}$. We will prove that modulo lower terms $$l_1\circ e_{\undtheta(m,n)}= l_2\circ e_{\undtheta(m,n)}= \overline{\Psi}^n_{m,n}$$ and that $l_3\circ e_{\undtheta(m,n)}=0$ (thus proving that $\overline{\Psi}^n_{m,n}\neq 0$).

\begin{enumerate}
\item The fact that $l_1\circ e_{\undtheta(m,n)}= \overline{\Psi}^n_{m,n}$ follows directly from the absorption property: ($e_{\undtheta(m,n-1)}\otimes \mathrm{id}_{B_sB_t})\circ e_{\undtheta(m,n)}=e_{\undtheta(m,n)}$.
\item Let us prove that $l_2\circ e_{\undtheta(m,n)}= \overline{\Psi}^n_{m,n}\in \Hom^{2n}_{\not < \theta(m-n,0)}(B_{\theta(m,n)},B_{\theta(m-n,0)}).$ 
We can apply relation \eqref{dotswap} on the two brown strands. Only the first term in the RHS of \eqref{dotswap} counts because the
third one is a lower term, and for the second one, use several times the Nil-Hecke relation \cite[(5.2)]{EW2} to move the $\alpha$ to the left (up to lower terms), and then we obtain a degree $2$ polynomial multiplied by a morphism of degree $2n-2$, which is zero by \Cref{boundsonmorphism}.
In the term that survives pulling back the brown dot through the six-valent vertex, we obtain the sum of two terms: $\overline{\Psi}^n_{m,n}$ and $\tau $, where $\tau$ is the following morphism.
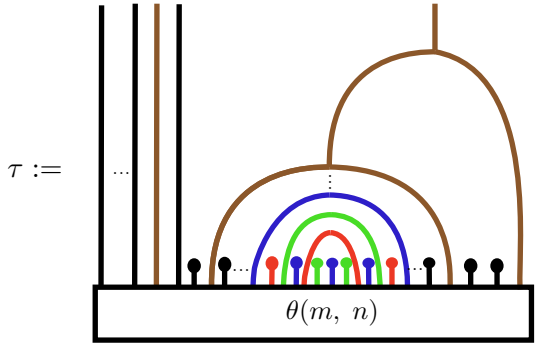
\begin{center}

\end{center}
We can apply \Cref{fac1} to show that $\tau$ is zero modulo lower terms. In fact, by removing the first dot we see that it factors through $BS(\undz)$, with $z\gtrdot \theta(m-n,0)$.


\item If $n>1$, the morphism $l_3\circ e_{\undtheta(m,n)}$ is zero modulo lower terms. In fact, we can remove the dot on the third strand to see that the map factors through a morphism of degree $2n-1$, and apply \Cref{fac1}.


\end{enumerate}


\subsubsection*{Case \texorpdfstring{$i<n$}{i<n}}

The proof in the general case follows the same lines of the case $i=n$. We indicate here the morphism $\Psi^i_{m,n}$ by a box labeled with label $\Psi^i_{m,n}$. 

	\begin{lemma}\label{fourlittleleaves}
		We have 
		\[\dim_{\mathbb{Q}} \Hom^{2i}_{\not < \theta(m-i,n-i)} (B_{\theta(m,n-1)}B_sB_t,B_{\theta(m-i,n-i)})=\begin{cases} 3 &\text{ if }i=1\\
		4 & \text{ if }i>1.\end{cases}.\]	
		and a basis of 	$\Hom^{2i}_{\not < \theta(m-i,n-i)} (B_{\theta(m,n-1)}B_sB_t,B_{\theta(m-i,n-i)})$ is given by the following four elements (where the third and forth diagrams have resp.  $i-2$ and $i-1$ arcs):
		\begin{center}
			$


\end{center}

\end{lemma}

\begin{proof}
	The proof is similar to \Cref{threelittleleaves} and we omit it. Notice that the third light leaf does not occur if $i=1$.
\end{proof}

By \Cref{boundsonmorphism} we know that 
\begin{equation}\label{dim2}
	\dim_{\mathbb{Q}} \Hom^{2i}(B_{\theta(m,n)},B_{\theta(m-i,n-i)})=1,
\end{equation}
and this space is generated by the four leaves in \Cref{threelittleleaves} (that we will call $l_1,\ldots l_4$) composed with the favorite projector $e_{\undtheta(m,n)}$. We have 
 \[l_1\circ e_{\undtheta(m,n)}= \overline{\Psi}^i_{m,n}.\] 
 We show that modulo lower terms we have $l_2\circ e_{\undtheta(m,n)}=l_3\circ e_{\undtheta(m,n)}=0$, and $l_4\circ e_{\undtheta(m,n)}= c\overline{\Psi}^i_{m,n}$ for some $c\in \bbQ$. (thus proving that $\overline{\Psi}^i_{m,n}\neq 0$):

To show $l_2\circ e_{\undtheta(m,n)}=0$ we remove the dot on the third strand and see that it factors through a morphism of degree $2i-1$, so we can apply \Cref{fac1}. The proof that $l_3\circ e_{\undtheta(m,n)}=0$ is similar: by removing the dot on the third strands we see that it factors through a morphism of degree $2i-1$, so we can apply again \Cref{fac1}.

It remains to consider the morphism $l_4\circ e_{\undtheta(m,n)}$. We can write it as the difference $l_4\circ e_{\undtheta(m,n)}=m_1-m_2$, where $m_1$ and $m_2$ are the following two morphisms. 
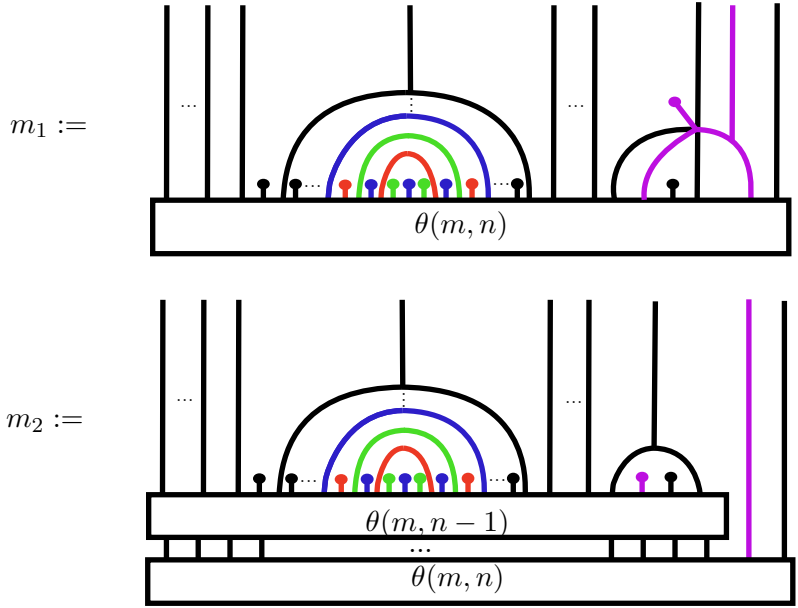
\begin{center}



\end{center}

In $m_1$, we can remove the dot from the purple strand and see that it factors through a morphism of degree $2i-1$ that must vanish modulo lower terms by \Cref{fac1}.
The second morphism is of the form $m_2=(\phi \otimes \Iden_{B_{st}})\circ e_{\undtheta(m,n)}$, where $\phi: B_{\theta(m,n-1)}\raw BS(\undtheta(m-i,n-i-1))$ is a morphism of degree $2i$.
By induction hypothesis and \Cref{boundsonmorphism}, modulo lower terms $\phi$ is a multiple of the only element in degree $2i$ of  $\ILL(\theta(m,n-1),\theta(m-i,n-i-1))$, which is $\bar{\Psi}^i_{m,n-1}$. It follows that $m_2$, and hence also $l_4\circ e_{\undtheta(m,n)}$ is a multiple of $\bar{\Psi}^i_{m,n}$.

\bibliographystyle{alpha}

\Address

\end{document}